\documentclass{article}

\usepackage{amssymb,bm}
\usepackage{amsmath}
\usepackage{amsthm}
\usepackage{graphicx}

\newtheorem{theorem}{{\sc Theorem}}[section]

\newtheorem{lemma}[theorem]{{\sc Lemma}}
\newtheorem{corollary}[theorem]{Corollary}
\newtheorem{remark}[theorem]{Remark}

\newtheorem{reflist}[theorem]{List}
\newtheorem{definition}[theorem]{Definition}

\newtheorem{conjecture}[theorem]{Conjecture}

\newcommand{\bb}[1]{\mathbb{ #1}}

\newcommand{\End}{\mathrm{End}}

\newcommand{\Sym}{\mathrm{Sym}}

\newcommand{\Skew}{\mathrm{Skew}}
\newcommand{\Span}{\mathrm{Span}}
\newcommand{\defeq}{{\buildrel\rm def\over=}}

\newcommand{\bra}[1]{\overline{#1}}

\newcommand{\rank}{\mathrm{rank}}

\newcommand{\hf}{\displaystyle\frac{1}{2}}
\newcommand{\nth}[1]{\displaystyle\frac{1}{#1}}

\newcommand{\Grad}{\nabla}
\newcommand{\Div}{\nabla \cdot}

\renewcommand{\Hat}[1]{\widehat{#1}}

\newcommand{\av}[1]{\langle #1 \rangle}

\def\XXint#1#2#3{{\setbox0=\hbox{$#1{#2#3}{\int}$ }
\vcenter{\hbox{$#2#3$ }}\kern-.6\wd0}}

\newcommand{\re}{\Re\mathfrak{e}}

\newcommand{\mat}[4]{\left[\begin{array}{cc}
\displaystyle{#1}&\displaystyle{#2}\\[1ex]
\displaystyle{#3}&\displaystyle{#4}\end{array}\right]}
\newcommand{\vect}[2]{\left[\begin{array}{c}
\displaystyle{#1}\\[1ex]\displaystyle{#2}\end{array}\right]}

\newcommand{\mc}{microstructure}
\newcommand{\mg}{microgeometr}

\newcommand{\WLOG}{without loss of generality}
\newcommand{\nbh}{neighborhood}
\newcommand{\IFF}{if and only if }

\newcommand{\complete}{complete}
\newcommand{\completeness}{completeness}

\newcommand{\Ga}{\alpha}
\newcommand{\Gb}{\beta}
\newcommand{\Gd}{\delta}

\newcommand{\Gve}{\varepsilon}

\newcommand{\Gvf}{\varphi}
\newcommand{\Gg}{\gamma}

\newcommand{\Gl}{\lambda}

\newcommand{\Gth}{\theta}

\newcommand{\Gs}{\sigma}

\newcommand{\GG}{\Gamma}

\newcommand{\GL}{\Lambda}

\bmdefine\BGa{\alpha}
\bmdefine\BGb{\beta}
\bmdefine\BGd{\delta}
\bmdefine\BGe{\epsilon}
\bmdefine\BGve{\varepsilon}
\bmdefine\BGf{\phi}
\bmdefine\BGvf{\varphi}
\bmdefine\BGg{\gamma}
\bmdefine\BGc{\chi}
\bmdefine\BGi{\iota}
\bmdefine\BGk{\kappa}
\bmdefine\BGl{\lambda}
\bmdefine\BGn{\eta}
\bmdefine\BGm{\mu}
\bmdefine\BGv{\nu}
\bmdefine\BGp{\pi}
\bmdefine\BGth{\theta}
\bmdefine\BGvth{\vartheta}
\bmdefine\BGr{\rho}
\bmdefine\BGvr{\varrho}
\bmdefine\BGs{\sigma}
\bmdefine\BGvs{\varsigma}
\bmdefine\BGt{\tau}
\bmdefine\BGj{\tau}
\bmdefine\BGu{\upsilon}
\bmdefine\BGo{\omega}
\bmdefine\BGx{\xi}
\bmdefine\BGy{\psi}
\bmdefine\BGz{\zeta}
\bmdefine\BGD{\Delta}
\bmdefine\BGF{\Phi}
\bmdefine\BGG{\Gamma}
\bmdefine\BGL{\Lambda}
\bmdefine\BGP{\Pi}
\bmdefine\BGT{\Theta}
\bmdefine\BGS{\Sigma}
\bmdefine\BGU{\Upsilon}
\bmdefine\BGO{\Omega}
\bmdefine\BGX{\Xi}
\bmdefine\BGY{\Psi}

\newcommand{\CA}{{\mathcal A}}
\newcommand{\CB}{{\mathcal B}}
\newcommand{\CC}{{\mathcal C}}

\newcommand{\CH}{{\mathcal H}}

\newcommand{\CL}{{\mathcal L}}
\newcommand{\CM}{{\mathcal M}}
\newcommand{\CN}{{\mathcal N}}

\newcommand{\CT}{{\mathcal T}}
\newcommand{\CU}{{\mathcal U}}
\newcommand{\CV}{{\mathcal V}}
\newcommand{\CW}{{\mathcal W}}

\bmdefine\BCA{{\mathcal A}}
\bmdefine\BCB{{\mathcal B}}
\bmdefine\BCC{{\mathcal C}}
\bmdefine\BCD{{\mathcal D}}
\bmdefine\BCE{{\mathcal E}}
\bmdefine\BCF{{\mathcal F}}
\bmdefine\BCG{{\mathcal G}}
\bmdefine\BCH{{\mathcal H}}
\bmdefine\BCI{{\mathcal I}}
\bmdefine\BCJ{{\mathcal J}}
\bmdefine\BCK{{\mathcal K}}
\bmdefine\BCL{{\mathcal L}}
\bmdefine\BCM{{\mathcal M}}
\bmdefine\BCN{{\mathcal N}}
\bmdefine\BCO{{\mathcal O}}
\bmdefine\BCP{{\mathcal P}}
\bmdefine\BCQ{{\mathcal Q}}
\bmdefine\BCR{{\mathcal R}}
\bmdefine\BCS{{\mathcal S}}
\bmdefine\BCT{{\mathcal T}}
\bmdefine\BCU{{\mathcal U}}
\bmdefine\BCV{{\mathcal V}}
\bmdefine\BCW{{\mathcal W}}
\bmdefine\BCX{{\mathcal X}}
\bmdefine\BCY{{\mathcal Y}}
\bmdefine\BCZ{{\mathcal Z}}

\bmdefine\Bzr{ 0}
\bmdefine\Ba{ a}
\bmdefine\Bb{ b}
\bmdefine\Bc{ c}
\bmdefine\Bd{ d}
\bmdefine\Be{ e}
\bmdefine\Bf{ f}
\bmdefine\Bg{ g}
\bmdefine\Bh{ h}
\bmdefine\Bi{ i}
\bmdefine\Bj{ j}
\bmdefine\Bk{ k}
\bmdefine\Bl{ l}
\bmdefine\Bm{ m}
\bmdefine\Bn{ n}
\bmdefine\Bo{ o}
\bmdefine\Bp{ p}
\bmdefine\Bq{ q}
\bmdefine\Br{ r}
\bmdefine\Bs{ s}
\bmdefine\Bt{ t}
\bmdefine\Bu{ u}
\bmdefine\Bv{ v}
\bmdefine\Bw{ w}
\bmdefine\Bx{ x}
\bmdefine\By{ y}
\bmdefine\Bz{ z}
\bmdefine\BA{ A}
\bmdefine\BB{ B}
\bmdefine\BC{ C}
\bmdefine\BD{ D}
\bmdefine\BE{ E}
\bmdefine\BF{ F}
\bmdefine\BG{ G}
\bmdefine\BH{ H}
\bmdefine\BI{ I}
\bmdefine\BJ{ J}
\bmdefine\BK{ K}
\bmdefine\BL{ L}
\bmdefine\BM{ M}
\bmdefine\BN{ N}
\bmdefine\BO{ O}
\bmdefine\BP{ P}
\bmdefine\BQ{ Q}
\bmdefine\BR{ R}
\bmdefine\BS{ S}
\bmdefine\BT{ T}
\bmdefine\BU{ U}
\bmdefine\BV{ V}
\bmdefine\BW{ W}
\bmdefine\BX{ X}
\bmdefine\BY{ Y}
\bmdefine\BZ{ Z}

\title{Lamination exact relations and their stability under homogenization}
\author{Yury Grabovsky}
\begin{document}
\maketitle
\begin{abstract}
  Relations between components of the effective tensors of composites that
  hold regardless of composite's microstructure are called exact
  relations. Relations between components of the effective tensors of all
  laminates are called lamination exact relations. The question of existence
  of sets of effective tensors of composites that are stable under lamination,
  but not homogenization was settled by Milton with an example in 3D
  elasticity. In this paper we discuss an analogous question for exact
  relations, where in a wide variety of physical contexts it is known (a
  posteriori) that all lamination exact relations are stable under
  homogenization. In this paper we consider 2D polycrystalline multi-field
  response materials and give an example of an exact relation that is stable
  under lamination, but not homogenization. We also shed some light on the
  surprising absence of such examples in most other physical contexts
  (including 3D polycrystalline multi-field response materials). The methods
  of our analysis are algebraic and lead to an explicit description (up to
  orthogonal conjugation equivalence) of all representations of formally real
  Jordan algebras as symmetric $n\times n$ matrices. For each representation
  we examine the validity of the 4-chain relation|a 4th degree polynomial
  identity, playing an important role in the theory of special Jordan algebras.
\end{abstract}


\section{Introduction}
\setcounter{equation}{0} 
\label{sec:intro}
The study of effective behavior of composite materials abounds with beautiful
formulas linking the components of the effective tensor and the tensor of
material properties of its constituents, when virtually nothing is known about
the \mc\ of the composite
\cite{hill63,kell64,levi,roha70,hash84,schu87,mish89b,mish89a,dvor90,dvor92,%
  dunn93,benv93,benv94,milg97} (see also a review by Milton
\cite{milt97}). The general theory of such formulas was developed in
\cite{grab98,gms00,grab03}. The idea was to identify all relations that are
preserved in all laminate \mc s, where two constituent materials are combined
in layers perpendicular to a given unit vector (lamination direction).
According to the general theory, reviewed in Section~\ref{app:toer}, each
relation preserved under lamination corresponds to a \emph{Jordan
  multialgebra}|a subspace of the space of symmetric matrices closed with
respect to \emph{several} Jordan multiplications. Mutations of Jordan algebras
\cite{koech99} are particular examples of Jordan multialgebras, where all
multiplications are ``internal'', i.e. coming from a single Jordan product in
a classical Jordan algebra. In the context of exact relations not every Jordan
multialgebra is a mutant of a classical one. The fundamental problem in the
theory of exact relations is whether every Jordan multialgebra gives rise to a
\mc-independent relation. If this is the case, then all lamination exact
relations are stable under homogenization. In \cite{gms00,grab03} we obtained
an algebraic condition on Jordan multialgebras sufficient for stability under
homogenization. In order to describe it we define \emph{completion} of a
Jordan multialgebra to be the set of all symmetric matrices in the smallest
associative multialgebra containing the original Jordan multialgebra. If the
Jordan multialgebra is \complete, i.e. equal to its completion, then the
corresponding lamination exact relation is stable under
homogenization. Mutations of classical special Jordan algebras do not seem to
posses superior completion properties compared to the general Jordan
multialgebras, and hence, will not be investigated as a special subclass in
this paper. The \completeness\ for classical special Jordan algebras is
related to the 4th degree polynomial identity, called the 4-chain relation,
via Cohn's theorem \cite{cohn54}. It is well-known that classical special
Jordan algebras do not have to be \complete. However, from the point of view
of the theory of composite materials the situation may be reversed if we
restrict our attention only to those Jordan multialgebras that arise in
physics. In particular, we will focus on the Jordan multialgebras
corresponding to the polycrystalline composites
\cite{grab98,grsa98,gms00}. These possess additional structure of a
representation space of $SO(3)$ or $SO(2)$ (for 2D or fiber-reinforced
composites, whose \mg y is completely determined by any 2D cross-section). The
interaction between the group representation structure of $SO(d)$ and the
multiplicative structure of Jordan multialgebras has been studied in a series
of papers \cite{sage01,sage05,sage05a,sage08} inspired by the theory of exact
relations.

While there is no general theorem that guarantees \completeness\ for $SO(3)$
or $SO(2)$-invariant Jordan multialgebras, all known physically relevant
examples of $SO(3)$ or $SO(2)$-invariant Jordan multialgebras are
\complete. Is it a coincidence or a new fundamental property specific to
rotationally invariant Jordan multialgebras? In this paper we exhibit a
physically motivated example in 2D featuring an $SO(2)$-invariant
\emph{in\complete} Jordan multialgebra and the corresponding lamination exact
relation that is not stable under homogenization (see (\ref{myex})). This
needs to be compared with an earlier result in \cite{gms00} that all
$SO(3)$-invariant Jordan multialgebras in the 3D analog of our example are
\complete, giving further support for the conjecture that all physically
relevant $SO(3)$-invariant Jordan multialgebras are \complete\ (see
Conjecture~\ref{conj:compl} below).

We remark that our example of $SO(2)$-invariant in\complete\ Jordan
multialgebra could lead to another example of a rank-1 convex function that is
not quasiconvex. While such an example has already been produced both in
Calculus of Variations \cite{sver92} and the theory of composite materials
\cite[Section~39.9]{mlt}, our example, coming from the theory of exact
relations would possess additional $SO(2)$ symmetry absent in the existing examples.

While the 3D version of our example, described in Section~\ref{sec:ex}, can be
completely analyzed (see Theorem~\ref{th:invJA}), the 2D version is more
involved. In the more difficult 2D case we only exhibit a specific family of
Jordan multialgebras generated by the subalgebras of $\Sym(\bb{R}^{n})$| the
Jordan algebra of all real symmetric $n\times n$ matrices. The question of
\completeness\ of these subalgebras has not been studied. The answer to this
question represents an algebraic contribution of this paper, which includes a
complete characterization of all \emph{faithful representations} of all
formally real Jordan algebras in $\Sym(\bb{R}^{n})$ up to orthogonal
conjugation. We note that all formally real Jordan algebras, up to a Jordan
algebra isomorphism, were described in \cite{jvnw34}, where an attempt was
made to build an algebraic foundation of quantum mechanics, where the
observables would be defined axiomatically, and not as operators on a Hilbert
space with the extraneous structure of associative multiplication.

The explicit characterization of Jordan subalgebras of $\Sym(\bb{R}^{n})$,
reveals that \completeness\ can indeed fail, but only in a few exceptional
cases. Adding more multiplications with respect to which the algebras have to
be closed makes these cases even more exceptional, so that in the relatively
low dimensions \completeness\ holds in general. Specifically, all Jordan
subalgebras of $\Sym(\bb{R}^{n})$ are \complete, if $n<8$. Even for general
$n$ \completeness\ is generic. A generic symmetric $n\times n$ matrix has $n$
distinct eigenvalues. Any subalgebra of $\Sym(\bb{R}^{n})$ containing such a
matrix must necessarily be \complete. This sheds light on the conspicuous
absence of in\complete\ Jordan multialgebras from all physically relevant
examples so far, since all of them, with the exception of the example in
Section~\ref{sec:ex}, are relatively low dimensional.

\section{General theory of exact relations}
\setcounter{equation}{0} 
\label{app:toer}
The standard references \cite{BLP,jko94} for the mathematical theory of
composite materials emphasize homogenization theorems and deal primarily with
conducting and elastic composites. In the case of periodic composites there is
an abstract Hilbert space framework \cite{kopa82,dfo86,MK,milt87a,milt87b,m}
that treats all coupled field composites (e.g.\@ piezo-electric, thermo-elastic,
etc.) from the common point of view. We refer to \cite{mlt} for a systematic
treatment of composites and exact relations based on these ideas. In this
paper we adopt the abstract Hilbert space formalism. Since our specific
question is algebraic in nature we give here an algebra-centered overview of
the general theory of exact relations. We refer the reader to \cite{gms00,mlt}
for a composite material-centered exposition.

A linear material response to applied fields (electric, elastic, thermal,
etc.) is described by a symmetric positive definite operator $L$ on a finite
dimensional inner product space $\CT$. If we are interested in polycrystalline
composites the space $\CT$ is also a representation space of the rotation
group $SO(3)$ or $SO(2)$, if we deal with fiber-reinforced composites,
whose \mc\ is determined by any cross-section perpendicular to the chosen ``fiber
direction''. In this case the inner product on $\CT$ is $SO(d)$-invariant, $d=2,3$.

In making a composite we choose a compact set of admissible materials
$U\subset\Sym^{+}(\CT)$, where $\Sym^{+}(\CT)$ is the set of symmetric
positive definite operators on $\CT$. In the case of polycrystalline
composites the set $U$ must be $SO(d)$-invariant. The effective tensor $L^{*}$
of a composite depends on the \mc. Varying the \mc\ over \emph{all}
possible ones we obtain the set $G(U)$ of corresponding effective tensors $L^{*}$ of
composites, called the G-closure of $U$.

Generically the set $G(U)$ has a non-empty interior in $\Sym^{+}(\CT)$, even
if $U$ consists of only 2 points. We are interested in special situations,
where $G(U)$ is a submanifold of $\Sym^{+}(\CT)$ of non-zero co-dimension.
\begin{definition}
  \label{def:ER}
  The submanifold $\bb{M}$ of $\Sym^{+}(\CT)$ is called \textbf{an exact
    relation} if the effective tensor $L^{*}$ of a composite will lie in
  $\bb{M}$, whenever constituent materials lie in $\bb{M}$, \emph{regardless
    of the microstructure}.
\end{definition}

\subsection{Lamination exact relations and Jordan multialgebras}
In order to identify these special exact relations cases we test G-closure set
by taking two arbitrary points
$\{L_{1},L_{2}\}\subset G(U)$ and forming a laminate---a composite
consisting of layers of material $L_{1}$ alternating with layers of material
$L_{2}$. Such a laminate is described by the direction of the lamination
$\Bn\in\bb{S}^{d-1}$ ($d=2$ or 3) and the volume fractions $\Gth_{1}$,
$\Gth_{2}=1-\Gth_{1}$ of $L_{1}$ and
$L_{2}$, respectively. For the simple laminate \mc\
there exists a beautiful explicit formula for the effective tensor \cite{m,mlt}
\begin{equation}
  \label{Wlam}
W_{\Bn}(L^{*};L_{0})=\Gth_{1}W_{\Bn}(L_{1};L_{0})+\Gth_{2}W_{\Bn}(L_{2};L_{0}),  
\end{equation}
where $L_{0}$ is an arbitrary ``reference material'' and the map
$W_{\Bn}(L,L_{0})$ has the form
\[
W_{\Bn}(L;L_{0})=[(L_{0}-L)^{-1}-\GG_{L_{0}}(\Bn)]^{-1}=
[I_{\CT}-(L_{0}-L)\GG_{L_{0}}(\Bn)]^{-1}(L_{0}-L),
\]
where $I_{\CT}$ is the identity map on $\CT$, and
$\GG_{L_{0}}(\Bn)\in\Sym(\CT)$ is an explicitly known function of
$\Bn\in\bb{S}^{d-1}$, whose specific definition is not important for our
purposes. We remark, that even though the formula (\ref{Wlam}) involves
$L_{0}$, the effective tensor $L^{*}$ defined by it is independent of $L_{0}$.
Restricting attention only to laminate \mc s we formulate the notion of
\emph{lamination exact relation}.
\begin{definition}
\label{def:lamER}
  A \textbf{lamination exact relation} is a submanifold $\bb{M}$ of positive
  co-dimension in $\Sym^{+}(\CT)$ such that the effective tensor $L^{*}$ of a
  laminate made with $\{L_{1},L_{2}\}\subset\bb{M}$ is in $\bb{M}$ for any
  choice of lamination direction $\Bn\in\bb{S}^{d-1}$ and volume faction
  $\Gth_{1}\in[0,1]$. 
\end{definition}
It is clear from (\ref{Wlam}) that $W_{\Bn}(\bb{M};L_{0})$ is a convex
subset of $\Sym(\CT)$ and at the same time a submanifold of $\Sym(\CT)$ of the same
co-dimension as $\bb{M}$. Therefore, $W_{\Bn}(\bb{M};L_{0})$ is an affine
subspace of $\Sym(\CT)$. However, choosing $L_{0}\in\bb{M}$ and observing that
$W(L_{0};L_{0})=0$ implies that the affine subspaces
$\Pi_{\Bn}=W_{\Bn}(\bb{M};L_{0})$ are linear subspaces of $\Sym(\CT)$. The function
$\GL_{\Bm,\Bn}=W_{\Bm}\circ W_{\Bn}^{-1}$ maps $\Pi_{\Bn}$ into $\Pi_{\Bm}$
diffeomorphically (at least in the \nbh\ of 0). We easily compute
\[
\GL_{\Bm,\Bn}(K)=[K^{-1}+\GG_{L_{0}}(\Bn)-\GG_{L_{0}}(\Bm)]^{-1}=[I_{\CT}-(\GG_{L_{0}}(\Bm)-\GG_{L_{0}}(\Bn))K]^{-1}K.
\]
We see that the differential of $\GL_{\Bm,\Bn}$ at $K=0$ is the identity map:
$d\GL_{\Bm,\Bn}(0)\xi=\xi$ for any $\xi\in T_{0}\Pi_{\Bn}$. That means that
the tangent spaces of $\Pi_{\Bn}$ and $\Pi_{\Bm}$ are the same. It follows
that $\Pi_{\Bn}=\Pi_{\Bm}$, since the submanifolds $\Pi_{\Bn}$ and
$\Pi_{\Bm}$, being subspaces of $\Sym(\CT)$, coincide with their tangent
spaces. Thus, $\Pi=W_{\Bn}(\bb{M};L_{0})$ is a well-defined subspace of
$\Sym(\CT)$, independent of $\Bn$. Expanding the map $\GL_{\Bn,\Bm}:\Pi\to \Pi$ in
powers of $K$ we obtain
\[
\GL_{\Bm,\Bn}(K)=KA_{\Bm,\Bn}K+KA_{\Bm,\Bn}KA_{\Bm,\Bn}K+KA_{\Bm,\Bn}KA_{\Bm,\Bn}KA_{\Bm,\Bn}K+\ldots,
\]
where $A_{\Bm,\Bn}=\GG_{L_{0}}(\Bm)-\GG_{L_{0}}(\Bn)$. We conclude that for $\bb{M}$ to be a
lamination exact relation it is necessary and sufficient that
\begin{equation}
  \label{Jmadef}
  KAK\in\Pi,\quad K\in\Pi,\quad A\in\CA=\Span\{\GG_{L_{0}}(\Bm)-\GG_{L_{0}}(\Bn):|\Bn|=1\}.
\end{equation}
It is easy to check that the subspace $\CA$ does not depend on $\Bm$
(regardless of what $\GG_{L_{0}}(\Bm)$ is).
The subspaces $\Pi$ satisfying (\ref{Jmadef}) are Jordan
multialgebras, since they are closed with respect to
a family of Jordan multiplications
\[
K_{1}*_{A}K_{2}=\hf(K_{1}AK_{2}+K_{2}AK_{1}),\qquad A\in\CA.
\]
To avoid any ambiguity we give the following definition.
\begin{definition}
  \label{def:sja}
A subspace $\Pi$ of an associative algebra $\CB$ is called a \textbf{Jordan
  $\CA$-multialgebra} if it is closed with respect to a collection of
Jordan multiplications\footnote{It would be more proper to use the term
  special Jordan $\CA$-multialgebra. In this paper we omit the qualifier
  ``special'', since we do not deal with the most general
Jordan algebras over a general field.}
\begin{equation}
  \label{GJprod}
X*_{A}Y=\hf(XAY+YAX),\qquad A\in\CA,  
\end{equation}
where $\CA$ is a subspace of $\CB$. A classical definition of a special Jordan
algebra corresponds to 1-dimensional subspaces $\CA$. A subspace $\Pi'$ of $\CB$ is called an
\textbf{associative $\CA$-multialgebra} if it is closed with respect to a collection
of multiplications $(X,Y)\mapsto XAY$, $A\in\CA$. The associative
$\CA$-multialgebra $\Pi'$ is called \textbf{symmetric} if $X\in\Pi'$ implies $X^{T}\in\Pi'$.
\end{definition}
We remark that the notion of mutation of a Jordan algebra \cite{koech99}
provides a family of examples of Jordan multialgebras (they can be called
``inner multialgebras''). If $\Pi_{0}$ is closed with respect to a \emph{single}
product $X*_{A_{0}}Y$ then it will always be closed with respect to a family
of multiplications (\ref{GJprod}), where $\CA=\{A_{0}BA_{0}:B\in\Pi_{0}\}$.
The question of comparing algebraic structures of an algebra and its mutation
is studied in \cite{koech99}, and is not the object of our investigation.

\subsection{Exact relations and \complete\  Jordan multialgebras}
Recall that Jordan multialgebras $\Pi$ correspond to the lamination
exact relations (see Definition~\ref{def:lamER}). The fundamental question in
the theory of exact relations is whether or not all lamination exact relations
are exact relations in the sense of Definition~\ref{def:ER}. This question was
investigated in \cite{grab98,gms00,grab03,mlt}, where simple algebraic
sufficient conditions have been formulated. To state this condition we define
the notion of \emph{\complete} Jordan $\CA$-multialgebra. Starting with a
Jordan $\CA$-multialgebra $\Pi$ we define $\mathfrak{A}(\Pi)$ to be the
smallest associative $\CA$-multialgebra in the sense of
Definition~\ref{def:sja} containing $\Pi$. The associative $\CA$-multialgebra
$\mathfrak{A}(\Pi)$ is necessarily symmetric in the sense of
Definition~\ref{def:sja}. We then define a \emph{completion} $\bra{\Pi}$
of $\Pi$ by
\begin{equation}
  \label{compdef}
\bra{\Pi}=\mathfrak{A}(\Pi)_{\rm sym},
\end{equation}
where
\[
V_{\rm sym}\defeq\,V\bigcap\Sym(\CT)=\{K+K^{T}:K\in V\}
\]
for a subspace $V\subset\End(\CT)$, such that $V^{T}=V$.
\begin{definition}
\label{def:comp}
  A Jordan $\CA$-multialgebra $\Pi\subset\Sym(\CT)$ is called \textbf{\complete} if
\begin{equation}
  \label{ER}
 \bra{\Pi}=\Pi,
\end{equation} 
We will say that $\Pi$ is \textbf{in\complete}, if $\Pi$ is not \complete.
\end{definition}
It is easy to see that the completion $\bra{\Pi}$ of $\Pi$ is the smallest
complete Jordan $\CA$-multialgebra containing $\Pi$.

According to the general theory \cite{gms00,grab03}, \complete\  Jordan
$\CA$-multialgebras $\Pi$ correspond to exact relations
$\bb{M}=W_{\Bn}^{-1}(\Pi;L_{0})$.

It is easy to show that $\Pi$ is \complete\  \IFF 3 and 4-chain
identities
 \begin{equation}
  \label{34chain}
K_{1}A_{1}K_{2}A_{2}K_{3}+K_{3}A_{2}K_{2}A_{1}K_{1}\in\Pi,\qquad  
K_{1}A_{1}K_{2}A_{2}K_{3}A_{3}K_{4}+K_{4}A_{3}K_{3}A_{2}K_{2}A_{1}K_{1}\in\Pi  
\end{equation}
hold for all $\{K_{1},K_{2},K_{3},K_{4}\}\subset\Pi$ and all
$\{A_{1},A_{2},A_{3}\}\subset\CA$. The proof is a straightforward modification
of the proof of Cohn's theorem \cite{cohn54,jaco68}. We remark that in the
case of classical Jordan algebras (one-dimensional subspace $\CA$) the 3-chain
relations hold automatically, while this is not so for Jordan
multialgebras. This is a crucial distinction between classical special Jordan
algebras and Jordan multialgebras, since a significant part of the
classical Jordan algebra theory is based on the triple product \cite{jaco68,koech99,mccr04}
\[
\{K_{1},K_{2},K_{3}\}=K_{1}K_{2}K_{3}+K_{3}K_{2}K_{1}.
\]
It is well-known that not every classical special Jordan algebra is
\complete. However, the examples given in textbooks are not subalgebras of
$\Sym(\bb{R}^{n})$|the space of all symmetric $n\times n$
matrices. Additionally, even if there are in\complete\ subalgebras of
$\Sym(\bb{R}^{n})$, it is still not clear if the presence of several
multiplications would not change the situation. In fact, we have already shown
\cite{gms00} that \emph{all} Jordan $SO(2)$ and $SO(3)$-invariant
multialgebras in important physical contexts, such as conductivity,
elasticity, piezo-electricity, etc. are \complete. In an attempt to resolve
this apparent contradiction between classical theory of Jordan algebras and
physically relevant examples we describe a class of $SO(2)$ and
$SO(3)$-invariant Jordan multialgebras.

\subsection{$SO(d)$-invariant Jordan multialgebras}
\label{sec:prelim}
In this section it is important that $\CT$ be a real finite dimensional
representation of $SO(d)$ ($d=2$ or 3). The irreducible representations (irreps) are
parametrized by non-negative integers, called weights. In the theory of
composite materials, where one is interested in coupled thermal, electric and
elastic properties of materials, the space $\CT$ can contain only the irreps
of weights 0, 1 and 2. The weight 0 irrep $W_{0}$ is a 1-dimensional space
with trivial group action. The weight 1 irrep $W_{1}$ is $\bb{R}^{d}$ with the
natural action of $SO(d)$, while the weight 2 irrep $W_{2}$ is the space
$\Sym_{0}(\bb{R}^{d})$ of symmetric trace-free $d\times d$ matrices with
$SO(d)$ acting by conjugation $A\mapsto RAR^{-1}$. Thus, all conceivable
coupled field physical problems would be accommodated by
\begin{equation}
  \label{Trep}
\CT=W_{0}\otimes\bb{R}^{n_{0}}\oplus W_{1}\otimes\bb{R}^{n_{1}}\oplus W_{2}\otimes\bb{R}^{n_{2}},  
\qquad n_{0}\ge 0,\ n_{1}\ge 0,\ n_{2}\ge 0
\end{equation}
for an appropriate choice of $n_{0}$, $n_{1}$ and $n_{2}$. The notation in
(\ref{Trep}) emphasizes that the group acts trivially on the second factors in
tensor products above. We assume that the representation $\CT$ is equipped
with a fixed $SO(d)$-invariant inner product. The space $\Sym(\CT)$ of
symmetric maps on $\CT$ has a natural action of $SO(d)$ and splits into the
direct sum of irreps up to weight 4. In addition to the representation $\CT$
the algebraic structure of our problem is determined by the choice of a
subrepresentation $\CA\subset\Sym(\CT)$, which must be isomorphic to
$W_{2}\oplus W_{4}$ as an $SO(d)$ module (or only $W_{2}$, if $n_{2}=0$).
The choice of such a subrepresentation is highly non-unique, and we assume
that a generic choice is made. The goal is to answer the fundamental question
of \completeness\ for $SO(d)$-invariant Jordan $\CA$-multialgebras.
\begin{conjecture}
\label{conj:compl}
Assume that $\CT$ is given by (\ref{Trep}) and
$\CA\subset\Sym(\CT)$ is isomorphic to $W_{2}\oplus W_{4}$ as an $SO(3)$
module. We conjecture that any rotationally invariant
$\CA$-multialgebra is \complete. 
\end{conjecture}
The rationale behind this conjecture is a positive result established in
\cite{gms00} (see also Section~\ref{sec:ex}), when $n_{0}=n_{2}=0$ and
$\CA=W_{2}\otimes\bb{R}I_{n_{1}}$, where $I_{n}$ denotes the $n\times n$
identity matrix. The conjecture is stated only for $SO(3)$-invariant Jordan
$\CA$-multialgebras, since this paper settles the $SO(2)$-invariant version of
the conjecture in the negative (see Section~\ref{sec:cex}). In addition to the
results of Section~\ref{sec:ex} we have also computed a complete list of
Jordan $\CA$-multialgebras for $n_{0}=n_{1}=n_{2}=1$ in \cite{gms00}. The
conjecture was then verified by hand for each individual Jordan
$\CA$-multialgebra.

\section{Case study: Multifield response composite materials}
\setcounter{equation}{0} 
\label{sec:ex}
Multifield linear response materials were considered in
\cite{mish89a,mish89b}. In this context $n$ coupled potential fields
$\BE=(\Grad\phi_{1},\ldots,\Grad\phi_{n})$ induce $n$ conjugate fluxes
$\BJ=(\Bj_{1},\ldots,\Bj_{n})$ satisfying
\[
\Div\Bj_{1}=\ldots=\Div\Bj_{n}=0.
\]
Thermoelectric properties of composites would fit in this context with $n=2$.
For general values of $n$ the space $\CT$ is a direct sum of $n$ copies of $\bb{R}^{d}$:
\[
\CT=\underbrace{\bb{R}^{d}\oplus\ldots\oplus\bb{R}^{d}}_{n}=V_{1}\otimes\bb{R}^{n},\quad d=2,3,
\]
where $V_{1}$ is the $SO(d)$ irrep of weight 1, i.e. $V_{1}=\bb{R}^{d}$ with
standard action of $SO(d)$. The space $\CA$ was computed in \cite{gms00}.
\[
\CA=W_{2}\otimes I_{n}\subset\Sym(\CT),
\]
where $W_{2}\subset\Sym(V_{1})$ is the $SO(d)$ irrep of weight 2, denoting the
space of symmetric, trace-free $d\times d$ matrices where $SO(d)$ acts by
conjugation. For the case $d=3$ all $SO(3)$-invariant Jordan
$\CA$-multialgebras have been computed in \cite{gms00}. The result is given in
Theorem~\ref{th:invJA} below.
\begin{theorem}
  \label{th:invJA}
  Let $\Pi$ be an $SO(3)$-invariant Jordan $\CA$-multialgebra in
  $\Sym(V_{1}\otimes\bb{R}^{n})$, where $\CA=W_{2}\otimes I_{n}$. Then there
  exists an associative subalgebra $\CB$ of $\End(\bb{R}^{n})$, such that
  $\CB^{T}=\CB$ and $\Pi=(\End(W_{1})\otimes\CB)_{{\rm sym}}$.
\end{theorem}
An immediate consequence of this theorem is that \emph{all} $SO(3)$-invariant
Jordan $\CA$-multialgebras are \complete.

When $d=2$ we can label points in $V_{1}=\bb{R}^{2}$ with standard action
of $SO(2)$, by complex numbers, so that the action of $R_{\Gth}\in SO(2)$ on
$z\in\bb{C}$ is given by $e^{i\Gth}z$. Here $R_{\Gth}$ denotes the rotation
through the angle $\Gth$ in counterclockwise (positive) direction. We write
\[
\CT=\bb{C}\otimes\bb{R}^{n}=\bb{C}^{n},\qquad
R_{\Gth}\cdot\Bu=e^{i\Gth}\Bu,\qquad\Bu\in\bb{C}^{n},\ R_{\Gth}\in SO(2).
\]
The $SO(2)$-invariant inner product on $\CT$ is $(\Bu,\Bv)_{\CT}=\re\av{\Bu,\Bv}$,
where $\av{\cdot,\cdot}$ is the standard Hermitean inner product on
$\bb{C}^{n}$. Every $K\in\End(\CT)=\End_{\bb{R}}(\bb{C}^{n})$ is uniquely
determined by two complex $n\times n$ matrices $X$ and $Y$ via its action on
$\Bu\in\bb{C}^{n}$:
\[
K\Bu=X\Bu+Y\bra{\Bu}.
\]
Therefore, we will write $K(X,Y)$ to identify elements of $\End(\CT)$.
We easily compute
\[
\Sym(\CT)=\{K(X,Y):X\in\CH(\bb{C}^{n}),\ Y\in\Sym(\bb{C}^{n})\},
\]
where $\CH(\bb{C}^{n})$ denotes the set of all complex Hermitean $n\times n$
matrices. In this notation
\[
\CA=\{K(0,zI_{n}):z\in\bb{C}\},\qquad R_{\Gth}\cdot K(X,Y)=K(X,e^{2\Gth}Y). 
\]
Therefore,
an arbitrary $SO(2)$ submodule of $\Sym(\CT)$ is given by
\[
\Pi=\Pi_{\CL,\CM}=\{(X,Y):X\in\CL\subset\CH(\bb{C}^{n}),\ Y\in\CM\subset\Sym(\bb{C}^{n})\},
\]
where $\CL$ can be any real subspace of $\CH(\bb{C}^{n}$ and $\CM$ can be any
complex subspace of $\Sym(\bb{C}^{n})$. If an $SO(2)$ submodule $\Pi$ is also
an $\CA$-multialgebra, then the subspaces $\CL$ and $\CM$ have to satisfy
\begin{equation}
    \label{2dJMA}
Y^{2}+XX^{T}\in\CM,\quad YX+XY^{*}\in\CL\text{ for all }X\in\CL,\ Y\in\CM.    
  \end{equation}
The 3-chain condition is equivalent to
\begin{equation}
  \label{3chain}
cX_{1}X_{2}^{T}X_{3}+\bra{c}X_{3}X_{2}^{T}X_{1}\in\CL,\text{ for all
}c\in\bb{C},\text{ and all }\{X_{1},X_{2},X_{3}\}\subset\CL.  
\end{equation}
The 4-chain condition is a lot more complicated.

The complete characterization of all solutions of (\ref{2dJMA}) is
unknown. However, some families of solutions can be easily identified.
We will focus on one such family, since it will lead us to in\complete\
$SO(2)$-invariant Jordan $\CA$-multialgebras, in contrast to
Theorem~\ref{th:invJA} in 3D.

Let $\Pi_{0}\subset\Sym(\bb{R}^{n})$ be a Jordan subalgebra of
$\Sym(\bb{R}^{n})$, i.e. $\Pi_{0}$ is a subspace of $\Sym(\bb{R}^{n})$
closed with respect to the Jordan product
\begin{equation}
  \label{Jprod}
A*B=\hf(AB+BA).  
\end{equation}
Then
\begin{equation}
  \label{zeroPi}
\CL=\{0\},\quad\CM=\Pi_{0}\otimes\bb{C}
\end{equation}
solves (\ref{2dJMA}). The 3-chain condition (\ref{3chain}) is therefore
trivially satisfied. The 4-chain condition from (\ref{34chain})
reduces to the classical 4-chain condition 
\begin{equation}
  \label{4chcl}
  Y_{1}Y_{2}Y_{3}Y_{4}+Y_{4}Y_{3}Y_{2}Y_{1}\in\Pi_{0}
  \text{ for all }\{Y_{1},Y_{2},Y_{3},Y_{4}\}\subset\Pi_{0}
\end{equation}

Our next goal is to characterize all Jordan subalgebras $\Pi_{0}$ of $\Sym(\bb{R}^{n})$
up to orthogonal conjugation. We observe that
all Jordan subalgebras of $\Sym(\bb{R}^{n})$ are formally real, i.e. if 
\[
X_{1}^{2}+\ldots+X_{m}^{2}=0,\quad m\ge 1,\ \{X_{1},\ldots,X_{m}\}\subset\Sym(\bb{R}^{n})
\]
then $X_{1}=\ldots=X_{m}=0$. The complete classification of
all formally real Jordan algebras, up to an isomorphism, has been achieved by
Jordan, von Neumann and Wigner in \cite{jvnw34}. Our goal is to go a little
further and classify all \emph{faithful representations} of formally real Jordan
algebras in $\Sym(\bb{R}^{n})$. We note that \completeness\ in the sense of
Definition~\ref{def:comp} is a property of the representation of an algebra,
not of the algebra itself. In fact, our results will produce an example of
two different faithful representations of the same algebra in
$\Sym(\bb{R}^{n})$, one of which is \complete, while the other is not.

\section{Formally real Jordan algebras}
\setcounter{equation}{0} 
\label{sec:frja}
In this section we review the complete classification by Jordan, von Neumann
and Wigner of all formally real Jordan algebras for the convenience of the
reader. In addition to the classification itself, we will also need many of
their intermediate results in \cite{jvnw34}.

The first set of statements (Theorem~\ref{th:jvnw1}) refers to an arbitrary
formally real Jordan algebra.
\begin{theorem}~
  \label{th:jvnw1}
  \begin{itemize}
  \item[(a)] There is a unique element $\bb{I}_{\Pi}\in\Pi$ such that
    $\bb{I}_{\Pi}*A=A$ for all $A\in\Pi$.
  \item[(b)] The subalgebra $\av{\bb{I}_{\Pi},A}$ generated by $A\in\Pi$ and
    $\bb{I}_{\Pi}$ contains pairwise
    orthogonal non-zero
    idempotents\footnote{i.e. $E_{\nu}*E_{\Gs}=\Gd_{\nu\Gs}E_{\nu}$,
      $\nu=1,\ldots,s$, $\Gs=1,\ldots,s$.}
    $E_{1},\ldots,E_{s}$ such that
    \begin{equation}
      \label{Adecomp}
E_{1}+\ldots+E_{s}=\bb{I}_{\Pi},\qquad\Gl_{1}E_{1}+\ldots+\Gl_{s}E_{s}=A      
    \end{equation}
for some $\{\Gl_{1},\ldots,\Gl_{s}\}\subset\bb{R}$.
\item[(c)] An idempotent is called unresolvable, if it cannot be written as a
  sum of two orthogonal non-zero idempotents. There exist unresolvable, pairwise
    orthogonal idempotents $E_{1},\ldots,E_{r}$ such that
    \begin{equation}
      \label{decomp1}
E_{1}+\ldots+E_{r}=\bb{I}_{\Pi}.      
    \end{equation}
    The resolution of unity (\ref{decomp1}) is not unique but the number $r$ of
    unresolvable pairwise orthogonal idempotents in it is always the same.
\item[(d)] If $\CA\subset\Pi$ is a proper ideal then $\Pi=\CA\oplus\CB$ as a direct
  sum of algebras, where $\CB=\{B\in\Pi:A*B=0,\text{ for all }A\in\CA\}$ is
  the complementary ideal.
\item[(e)] Any algebra $\Pi$ is a direct (orthogonal) sum of simple algebras
  \begin{equation}
    \label{sdecomp}
\Pi=\Pi_{1}\oplus\ldots\oplus\Pi_{m},    
  \end{equation}
where each component $\Pi_{j}$ is identified with a minimal proper ideal in $\Pi$.
  \end{itemize}
\end{theorem}
The second set of results (Theorem~\ref{th:jvnw2}) refers to a simple formally
real Jordan algebra $\Pi$, in which the decomposition of unity (\ref{decomp1})
has been chosen. Let
\[
\CM^{\rho\Gs}=\left\{A\in\Pi:E_{\tau}*A=\hf(\Gd_{\rho\tau}+\Gd_{\Gs\tau})A
\text{ for all }\tau=1,\ldots,r\right\},\quad\rho,\Gs=1,\ldots,r.
\]
Clearly, $\CM^{\rho\Gs}=\CM^{\Gs\rho}$.
\begin{theorem}~
  \label{th:jvnw2}
  \begin{itemize}
  \item[(a)] 
$\displaystyle
\Pi=\bigoplus_{1\le\rho\le\Gs\le r}\CM^{\rho\Gs},
$
as a direct sum of vector spaces.
\item[(b)] $\dim\CM^{\Gs\Gs}=1$, $p=\dim\CM^{\rho\Gs}$ is independent of
  $\Gs$ and $\rho$, as long as $\Gs\not=\rho$.
\item[(c)] The spaces $\CM^{\rho\Gs}$ have bases
  $\{X_{1}^{\rho\Gs},\ldots,X_{p}^{\rho\Gs}\}$ such that
\[
X_{\mu}^{\rho\Gs}*X_{\nu}^{\rho\Gs}=\Gd_{\mu\nu}(E_{\rho}+E_{\Gs}),\quad
\mu,\nu=1,\ldots,p,\ 1\le\rho<\Gs\le r.
\] 
  \end{itemize}
\end{theorem}
The third result (Theorem~\ref{th:jvnw3}) is the well-known characterization
of all formally real simple Jordan algebras.
\begin{theorem}
  \label{th:jvnw3}
If $\Pi$ is a formally real
simple Jordan algebra then it is isomorphic to one of the following algebras
classified according to the number of indecomposable idempotents in the
resolution of unity (\ref{decomp1})
\begin{itemize}
\item $r=1$, $\bb{R}$
\item $r=2$, ``spin factors'' $\mathfrak{S}_{n}$, $n\ge 3$ defined by
$
\mathfrak{S}_{n}=\Span\{I_{\mathfrak{S}_{n}},s_{1},\ldots,s_{n-1}\},
$
where $I_{\mathfrak{S}_{n}}*I_{\mathfrak{S}_{n}}=I_{\mathfrak{S}_{n}}$, $I_{\mathfrak{S}_{n}}*s_{j}=s_{j}$, $s_{i}*s_{j}=\Gd_{ij}I_{\mathfrak{S}_{n}}$.
\item $r\ge 3$, $\Sym(\bb{R}^{r})$, $\mathfrak{H}(\bb{C}^{r})$,
  $\mathfrak{H}(\bb{H}^{r})$, denoting real symmetric $r\times r$
  matrices, complex Hermitean $r\times r$
  matrices and quaternionic Hermitean $r\times r$ matrices, respectively.
\item 27-dimensional exceptional Albert algebra $\mathfrak{M}_{3}^{8}$
  ($r=3$). It has no non-trivial representations in $\Sym(\bb{R}^{n})$.
\end{itemize}
\end{theorem}
It was shown by Albert \cite{albert34} that $\mathfrak{M}_{3}^{8}$ cannot be
identified as a subspace of an associative algebra that is closed under the
multiplication (\ref{Jprod}).

\begin{remark}
  \label{rem:jvnw}
It will be useful to list the value of the invariant $p=\dim\CM^{\rho\Gs}$,
defined in Theorem~\ref{th:jvnw2}, for all the simple algebras from
Theorem~\ref{th:jvnw3}.
\begin{itemize}
\item $p=1$ for $\Sym(\bb{R}^{r})$, $r\ge 1$ ($\Sym(\bb{R}^{2})\cong\mathfrak{S}_{3}$),
\item $p=2$ for $\mathfrak{H}(\bb{C}^{r})$, $r\ge 2$ 
($\mathfrak{H}(\bb{C}^{2})\cong\mathfrak{S}_{4}$),
\item $p=3$ for $\mathfrak{S}_{5}$
\item $p=4$ for $\mathfrak{H}(\bb{H}^{r})$, $r\ge 2$ 
($\mathfrak{H}(\bb{H}^{2})\cong\mathfrak{S}_{6}$),
\item $p\ge 5$ for $\mathfrak{S}_{p+2}$.
\end{itemize}
For $\mathfrak{M}_{3}^{8}$ $p=8$.
\end{remark}

\section{The structure of Jordan subalgebras of $\Sym(\bb{R}^{n})$}
\setcounter{equation}{0} 
\label{sec:ssc}
\subsection{The non-singular Jordan algebras}
From now on $\Pi$ refers to a subalgebra of $\Sym(\bb{R}^{n})$, i.e. a
subspace closed with respect to the multiplication (\ref{Jprod}).
Let 
\[
\CN(\Pi)=\{u\in\bb{R}^{n}: Au=0\text{ for all }A\in\Pi\}.
\]     
Then, in the basis which is the union of the bases for $V=\CN(\Pi)^{\perp}$ and
$\CN(\Pi)$
\begin{equation}
  \label{red0}
\Pi=\left\{\mat{A}{0}{0}{0}:A\in\Pi_{0}\subset\Sym(V)\right\},  
\end{equation}
where $\Pi_{0}$ is a Jordan subalgebra of $\Sym(V)$, for which $\CN(\Pi_{0})=\{0\}$.
\begin{definition}
  We say that the Jordan algebra $\Pi$ is \textbf{non-singular} if $\CN(\Pi)=\{0\}$.
\end{definition}
Any algebra $\Pi$ is non-singular on $V=\CN(\Pi)^{\perp}$.  Hence, \WLOG, we
may assume that $\Pi$ is non-singular.
  \begin{lemma}
\label{lem:nonsing}
Suppose that the Jordan algebra $\Pi\subset\Sym(\bb{R}^{n})$ is non-singular.
Then, $I_{n}\in\Pi$, where $I_{n}$ is the $n\times n$ identity matrix.
  \end{lemma}
  \begin{proof}
By Theorem~\ref{th:jvnw1} there exists the algebra identity
$\bb{I}_{\Pi}\in\Pi$. Then $\bb{I}_{\Pi}^{2}=\bb{I}_{\Pi}$. Therefore,
the symmetric matrix $\bb{I}_{\Pi}$ may have eigenvalues that are either 1 or 0. Suppose
$u\in\bb{R}^{n}$ is an eigenvector of $\bb{I}_{\Pi}$ with eigenvalue
zero. Then, by assumption, there exists $A\in\Pi$, such that
$v=Au\not=0$. Applying the equality $2A=A\bb{I}_{\Pi}+\bb{I}_{\Pi}A$ to the
vector $u$ we obtain $2v=\bb{I}_{\Pi}v$. This contradicts the fact that 2 is
not an eigenvalue for $\bb{I}_{\Pi}$. Thus, $\bb{I}_{\Pi}$ may not have an
eigenvalue 0. Therefore, $\bb{I}_{\Pi}=I_{n}$.
\end{proof}
Let $A\in\Pi$. Suppose that its eigenvalues are
$\{\Gl_{1},\ldots,\Gl_{s}\}\subset\bb{R}$ and the corresponding eigenspaces
are $V_{\Ga}$, $\Ga=1,\ldots,s$. By Theorem~\ref{th:jvnw1} the idempotents
$E_{\Ga}$ in (\ref{Adecomp}) must be orthogonal projections $P_{V_{\Ga}}$ onto
the eigenspaces $V_{\Ga}$ of $A$. We will call them spectral
projections. Thus, for any $A\in\Pi$ all of its spectral projections
$P_{V_{\Ga}}$ must also be in $\Pi$, and
\begin{equation}
  \label{spec}
  A=\sum_{\Ga=1}^{s}\Gl_{\Ga}P_{V_{\Ga}},\qquad V_{1}\oplus\ldots\oplus V_{s}=\bb{R}^{n},
\end{equation}

\subsection{Splitting of Jordan algebras}
\begin{definition}
  \label{def:split}
  We say that the Jordan algebra $\Pi\subset\Sym(\bb{R}^{n})$ \textbf{splits}
  over the orthogonal decomposition $\bb{R}^{n}=V_{1}\oplus\ldots\oplus
  V_{m}$, if $\Pi=\Pi_{1}\oplus\ldots\oplus\Pi_{m}$, is a direct sum of Jordan
  algebras, where
\[
\Pi_{\Ga}=\{A\in\Pi:Aw=0\text{ for all }w\in V_{\Ga}^{\perp}\},\quad\Ga=1,\ldots,m.
\]
\end{definition}
If $\Pi$ splits over $\bb{R}^{n}=V_{1}\oplus\ldots\oplus V_{m}$, then in the basis, which
is the union of bases for $V_{\Ga}$, $\Ga=1,\ldots,m$ the algebra $\Pi$ has
the form
\[
\Pi=\left\{\left[
  \begin{array}{ccc}
    A_{1}&&0\\
    &\ddots&\\
    0&&A_{m}
  \end{array}\right]:A_{\Ga}\in\Pi_{\Ga}^{0}\subset\Sym(V_{\Ga}),\ \Ga=1,\ldots,m\right\}.
\]
In other words all matrices in $\Pi$ have block-diagonal structure, with
independent blocks $A_{\Ga}$ taken from subalgebras
$\Pi_{\Ga}^{0}\subset\Sym(V_{\Ga})$ that are isomorphic to $\Pi_{\Ga}$.
Clearly, each $\Pi_{\Ga}\subset\Pi$ is an ideal. Conversely, if
$\CA\subset\Pi$ is an ideal, then, according to Theorem~\ref{th:jvnw1}, there
exists a complementary ideal
\[
\CB=\{B\in\Pi:A*B=0\text{ for all }A\in\CA\},
\]
such that $\Pi=\CA\oplus\CB$, as a direct sum of algebras.
\begin{lemma}
  \label{lem:ideal}
Let $\CA$ and $\CB$ be the complementary pair of ideals in $\Pi$. Then
there exists $V\subset\bb{R}^{n}$ such that $\Pi=\CA\oplus\CB$ is the
splitting of $\Pi$ over $\bb{R}^{n}=V\oplus V^{\perp}$.
\end{lemma}
\begin{proof}
  Let $V=\CN(\CB)$. The subspace $V$ cannot be all of $\bb{R}^{n}$, since
  otherwise $\CB=\{0\}$ and $\CA=\Pi$. If $V=\{0\}$, then $\CB$ is a
  non-singular algebra and $I_{n}\in\CB$, by Lemma~\ref{lem:nonsing}. But
  then $\CB=\Pi$ and $\CA=\{0\}$. The subspace $V^{\perp}$ is invariant for
  $\CB$, since all matrices in $\CB$ are symmetric. Thus, $\CB$ is isomorphic to
a subalgebra $\CB_{0}$ of $\Sym(V^{\perp})$ defined by the restriction of $\CB$ on
$V^{\perp}$. By our construction $\CN(\CB_{0})=\{0\}$ and, applying
Lemma~\ref{lem:nonsing} to $\CB_{0}$ we conclude that
$P_{V^{\perp}}\in\CB$. Now, let $B\in\Pi$ be such that $Bv=0$ for all $v\in
V$. Then $P_{V^{\perp}}*B=B$. Thus, $B\in\CB$, since $\CB$ is an ideal. We
have now proved that $\CB=\{B\in\Pi:Bv=0\text{ for all }v\in\CN(\CB)\}$. To finish
the proof of the lemma we need to show that $\CN(\CA)=V^{\perp}$. First we
observe that $V$ (and therefore $V^{\perp}$) is an invariant subspace for
$\CA$. Indeed, for any $v\in V$ and any $A\in\CA$ we have 
\[
P_{V^{\perp}}(Av)=2(P_{V^{\perp}}*A)v-A(P_{V^{\perp}}v)=0,
\]
since $P_{V^{\perp}}\in\CB$. Also, for any $w\in V^{\perp}$ and any $A\in\CA$ we have
\[
Aw=P_{V^{\perp}}(Aw)=2(P_{V^{\perp}}*A)w-A(P_{V^{\perp}}w)=-Aw.
\]
Thus, $V^{\perp}\subset\CN(\CA)$. To prove the reverse inclusion we
observe that $P_{V}*B=0$ for all $B\in\CB$. Indeed, $B*P_{V^{\perp}}=B$ for
all $B\in\CB$, therefore $B*P_{V}=B*(I_{n}-P_{V^{\perp}})=B-B=0$. It follows
that $P_{V}\in\CA$. Thus, if $Ax=0$ for all $A\in\CA$, then $P_{V}x=0$ and
$x\in V^{\perp}$. 
\end{proof}
\begin{corollary}
\label{cor:JAstruct}
  Let (\ref{sdecomp}) be the decomposition of $\Pi$ into a direct sum of
  simple algebras. Then, there exist pairwise orthogonal subspaces
  $V_{1},\ldots,V_{m}$ of $\bb{R}^{n}$ such that (\ref{sdecomp}) is the
  splitting of $\Pi$ over the orthogonal decomposition
  $\bb{R}^{n}=V_{1}\oplus\ldots\oplus V_{m}$. Thus,
  \begin{equation}
    \label{JAstruct}
\Pi=\left\{\left[
  \begin{array}{ccc}
    A_{1}&&0\\
    &\ddots&\\
    0&&A_{m}
  \end{array}\right]:A_{\Ga}\in\Pi_{\Ga}\subset\Sym(V_{\Ga}),\ \Ga=1,\ldots,m\right\},
  \end{equation}
where $\Pi_{\Ga}\subset\Sym(V_{\Ga})$, $\Ga=1,\ldots,m$ are simple Jordan algebras.
\end{corollary}

\subsection{Irreducible Jordan algebras}
According to Corollary~\ref{cor:JAstruct} we need to understand the structure
of simple non-singular Jordan algebras $\Pi\subset\Sym(\bb{R}^{n})$.
As in the theory of associative algebras it is important whether or not there
is a common invariant subspace for all matrices $A\in\Pi$. Any Jordan algebra
is completely reducible in the sense that there exists an orthogonal
decomposition 
\[
\bb{R}^{n}=V_{1}\oplus\ldots\oplus V_{k},
\] 
such that all subspaces $V_{\Ga}$ are invariant for all matrices $A\in\Pi$ and
they do not contain any smaller proper invariant subspaces for $\Pi$.
\begin{definition}
  \label{def:ssja}
A Jordan subalgebra of $\Sym(\bb{R}^{n})$ is called \textbf{irreducible} if it does
not have any proper invariant subspaces in $\bb{R}^{n}$.
\end{definition}
We remark that an irreducible Jordan subalgebra of $\Sym(\bb{R}^{n})$ is just
a faithful irreducible representation of a simple formally real Jordan algebra.
Our goal is to describe the structure of an arbitrary simple non-singular subalgebra
of $\Sym(\bb{R}^{n})$ in terms of irreducible algebras.
\begin{lemma}
  \label{lem:simptoir}
Let $\Pi$ be a simple non-singular subalgebra of $\Sym(\bb{R}^{n})$ and
$\bb{R}^{n}=V_{1}\oplus\ldots\oplus V_{k}$ be an orthogonal
decomposition of $\bb{R}^{n}$ into the sum of irreducible invariant
subspaces. Let
\[
\Pi_{\Ga}=\{P_{V_{\Ga}}AP_{V_{\Ga}}:A\in\Pi\}\subset\Sym(V_{\Ga}),\quad \Ga=1,\ldots,k.
\]
Then $\Pi_{\Ga}\subset\Sym(V_{\Ga})$, $\Ga=1,\ldots,k$ are irreducible Jordan
algebras in the sense of Definition~\ref{def:ssja}. Moreover, the maps
$T_{\Ga\Gb}:\Pi_{\Ga}\to\Pi_{\Gb}$ defined as
\[
T_{\Ga\Gb}A=P_{V_{\Gb}}KP_{V_{\Gb}},\qquad A=P_{V_{\Ga}}KP_{V_{\Ga}},\quad K\in\Pi
\]
are Jordan algebra isomorphisms.
\end{lemma}
\begin{proof}
Let $A\in\Pi_{\Ga}$ and suppose $\{K_{1},K_{2}\}\subset\Pi$ are such that
$A=P_{V_{\Ga}}K_{1}P_{V_{\Ga}}=P_{V_{\Ga}}K_{2}P_{V_{\Ga}}$. Then
$P_{V_{\Ga}}(K_{1}-K_{2})P_{V_{\Ga}}=0$. Let
\[
\CA=\{K\in\Pi:P_{V_{\Ga}}KP_{V_{\Ga}}=0\}.
\]
Let us show that $\CA$ is an ideal in $\Pi$. Let $X\in\Pi$ and $K\in\CA$ be
arbitrary. For any $x\in\bb{R}^{n}$ we have $x=v+w$, where $v\in V_{\Ga}$ and
$w\in V_{\Ga}^{\perp}$. Obviously $P_{V_{\Ga}}(X*K)P_{V_{\Ga}}w=0$. We also
have
\[
2P_{V_{\Ga}}(X*K)P_{V_{\Ga}}v=P_{V_{\Ga}}KXv,
\]
since $Kv=0$ for any $v\in V_{\Ga}$. But $Xv\in V_{\Ga}$, since $V_{\Ga}$ is
an invariant subspace for $\Pi$, and thus, $KXv=0$. It follows that
$X*K\in\CA$. The ideal $\CA$ cannot be $\Pi$, since in that case
$V_{\Ga}\subset\CN(\Pi)$, contrary to the assumption that $\Pi$ is a
non-singular algebra. Hence, $\CA=\{0\}$. Hence, for any $A\in\Pi_{\Ga}$ there
is a unique $K\in\Pi$ for which $A=P_{V_{\Ga}}KP_{V_{\Ga}}$. Moreover,
for any $\{K_{1},K_{2}\}\subset\Pi$ we have
\[
(P_{V_{\Ga}}K_{1}P_{V_{\Ga}})*(P_{V_{\Ga}}K_{2}P_{V_{\Ga}})=P_{V_{\Ga}}(K_{1}*K_{2})P_{V_{\Ga}}.
\]
Indeed, the invariance of $V_{\Ga}$ implies that for any $v\in V_{\Ga}$
\[
P_{V_{\Ga}}K_{1}P_{V_{\Ga}}P_{V_{\Ga}}K_{2}v=P_{V_{\Ga}}K_{1}K_{2}v=K_{1}K_{2}v.
\]
It follows that $\Pi_{\Ga}\subset\Sym(V_{\Ga})$ are irreducible Jordan algebras. 
The maps $T_{\Ga\Gb}$ are well-defines and preserve Jordan multiplication
\[
T_{\Ga\Gb}(A_{1}*A_{2})=P_{V_{\Gb}}(K_{1}*K_{2})P_{V_{\Gb}}=(P_{V_{\Gb}}K_{1}P_{V_{\Gb}})*(P_{V_{\Gb}}K_{2}P_{V_{\Gb}})=(T_{\Ga\Gb}A_{1})*(T_{\Ga\Gb}A_{2}).
\]
By construction, the maps $T_{\Ga\Gb}$ are surjective. Let us show that they
are also injective. If $A\in\Pi_{\Ga}$ and $A\not=0$, then there is a unique
$K\in\Pi$ such that $A=P_{V_{\Ga}}KP_{V_{\Ga}}$. Clearly, $K\not=0$. If
$T_{\Ga\Gb}A=P_{V_{\Gb}}KP_{V_{\Gb}}=0$, then $K'=0$ and $K\not=K'$ satisfy
$P_{V_{\Gb}}KP_{V_{\Gb}}=P_{V_{\Gb}}K'P_{V_{\Gb}}$ in contradiction of
uniqueness. 
\end{proof}
Hence, we have proved the structure theorem for simple non-singular Jordan
subalgebras of $\Sym(\bb{R}^{n})$. 
\begin{theorem}
  \label{th:str}
  Let $\Pi\subset\Sym(\bb{R}^{n})$ be a non-singular simple Jordan
  algebra. Then there exists an orthonormal basis (o.n.b.) of $\bb{R}^{n}$ in
  which the algebra $\Pi$ has the form
\begin{equation}
  \label{simpleJA}
\Pi=\left\{\left[
  \begin{array}{cccc}
    A&&&0\\
    &T_{1}A&&\\
    &&\ddots&\\
    0&&&T_{k-1}A
  \end{array}\right]:A\in\Pi_{0}\subset\Sym(\bb{R}^{N})\right\},  
\end{equation}
where $\Pi_{0}\subset\Sym(\bb{R}^{N})$ is an irreducible Jordan algebra and
$T_{1},\ldots,T_{k-1}$ are Jordan algebra isomorphisms. 
\end{theorem}
Thus, we have reduced the problem of description of all subalgebras of
$\Sym(\bb{R}^{n})$ to the problem of characterization of all irreducible
representations of simple Jordan algebras and their isomorphisms.

The problem of \completeness\ of $\Pi$ can also be restated.  It is now clear
that for a Jordan algebra to be \complete\  in the sense of
Definition~\ref{def:comp} it is necessary and sufficient for each of its
simple components $\Pi_{\Ga}$, to be \complete. The latter condition will be
satisfied \IFF the irreducible algebras $\Pi_{0}$ and
$\Pi_{\Ga}=T_{\Ga}\Pi_{0}$, $\Ga=1,\ldots,k-1$, in the representation
(\ref{simpleJA}) are \complete\  and Jordan isomorphisms $T_{\Ga}$
satisfy
\begin{equation}
  \label{T4chain}
  T_{\Ga}(A_{1}A_{2}A_{3}A_{4}+A_{4}A_{3}A_{2}A_{1})=
T_{\Ga}A_{1}T_{\Ga}A_{2}T_{\Ga}A_{3}T_{\Ga}A_{4}+T_{\Ga}A_{4}T_{\Ga}A_{3}T_{\Ga}A_{2}T_{\Ga}A_{1}.
\end{equation}

\section{The characterization of irreducible Jordan algebras}
\setcounter{equation}{0} 
\label{sec:iralg}
To complete the characterization of all Jordan subalgebras $\Sym(\bb{R}^{n})$ we
need to describe all irreducible Jordan algebras up to an orthogonal
equivalence and compute their Jordan isomorphisms. In this section we
take on the former problem.

\subsection{Block decomposition}
Let $\mathfrak{J}$ be a simple formally real Jordan algebra. Let
$\Phi:\mathfrak{J}\to\Sym(\bb{R}^{n})$ be an irreducible representation of
$\mathfrak{J}$. 
According to Theorem~\ref{th:jvnw2} there exist $r$ indecomposable orthogonal
idempotents $\{E_{1},\ldots,E_{r}\}\subset\mathfrak{J}$. Their images under
the representation $\Phi$ must necessarily be orthogonal projectors
onto the mutually orthogonal subspaces $W_{1},\ldots,W_{\Ga}$ such that
$\bb{R}^{n}=W_{1}\oplus\ldots\oplus W_{r}$. The algebra $\Pi=\Phi(\mathfrak{J})$ can be
represented as a direct sum $\Pi=\bigoplus_{1\le\rho\le\Gs\le r}\CM^{\rho\Gs}$
(block decomposition), where
\[
\CM^{\Ga\Ga}=\bb{R}P_{W_{\Ga}},\quad P_{W_{\Ga}}=\Phi(E_{\Ga}),\quad\Ga=1,\ldots,r, 
\]
and
\[
\CM^{\Ga\Gb}=\{P_{W_{\Ga}}AP_{W_{\Gb}}+P_{W_{\Gb}}AP_{W_{\Ga}}:A\in\Pi\},\quad
1\le\Ga<\Gb\le r.
\]
The spaces $\CM^{\Ga\Gb}$ have the same dimension $p$ for all $1\le\Ga<\Gb\le r$
and have the basis 
\begin{equation}
  \label{Xabm}
  X_{\mu}^{\Ga\Gb}=P_{W_{\Ga}}A_{\mu}P_{W_{\Gb}}+P_{W_{\Gb}}A_{\mu}P_{W_{\Ga}}
\end{equation}
such that
\begin{equation}
  \label{spinmult}
  X_{\mu}^{\Ga\Gb}*X_{\nu}^{\Ga\Gb}=\Gd_{\mu\nu}(P_{W_{\Ga}}+P_{W_{\Gb}}),\quad
\mu,\nu=1,\ldots,p.
\end{equation}
In particular,
\[
(X_{\mu}^{\Ga\Gb})^{2}=P_{W_{\Ga}}A_{\mu}P_{W_{\Gb}}A_{\mu}P_{W_{\Ga}}+
P_{W_{\Gb}}A_{\mu}P_{W_{\Ga}}A_{\mu}P_{W_{\Gb}}=P_{W_{\Ga}}+P_{W_{\Gb}}.
\]
It follows that 
\[
P_{W_{\Ga}}A_{\mu}P_{W_{\Gb}}A_{\mu}P_{W_{\Ga}}=P_{W_{\Ga}},\qquad
P_{W_{\Gb}}A_{\mu}P_{W_{\Ga}}A_{\mu}P_{W_{\Gb}}=P_{W_{\Gb}}.
\]
From the first equality we obtain that
\[
\dim(W_{\Ga})=\rank(P_{W_{\Ga}})=\rank(P_{W_{\Ga}}A_{\mu}P_{W_{\Gb}}A_{\mu}P_{W_{\Ga}})\le\dim(W_{\Gb}),
\]
while from the second one we obtain
\[
\dim(W_{\Gb})=\rank(P_{W_{\Gb}})=\rank(P_{W_{\Gb}}A_{\mu}P_{W_{\Ga}}A_{\mu}P_{W_{\Gb}})\le\dim(W_{\Ga}).
\]
Thus, $\dim(W_{\Ga})=d=n/r$ for all $\Ga=1,\ldots,r$.

\subsection{Structure of $\CM^{\Ga\Gb}$-blocks}
In order to continue, it will be convenient to fix arbitrary orthonormal bases for
spaces $W_{\Ga}$ and work with the space $\CU^{\Ga\Gb}$ of $d\times d$ matrices representing the
$\Ga\Gb$-block of $\CM^{\Ga\Gb}$. The $n\times n$ symmetric matrices
$X_{\mu}^{\Ga\Gb}\in\CM^{\Ga\Gb}$ given by (\ref{Xabm}) can be described
by their $\Ga\Gb$ $d\times d$ blocks
$\Hat{X}_{\mu}^{\Ga\Gb}\in\CU^{\Ga\Gb}$. When $\Ga\not=\Gb$ the relation
(\ref{spinmult}) for $X_{\mu}^{\Ga\Gb}$ can be written in terms of matrices
$\Hat{X}_{1}^{\Ga\Gb},\ldots,\Hat{X}_{p}^{\Ga\Gb}$ as follows
\[
\Hat{X}_{\mu}^{\Ga\Gb}(\Hat{X}_{\mu}^{\Ga\Gb})^{T}=I_{d},\qquad \Hat{X}_{\mu}^{\Ga\Gb}(\Hat{X}_{\nu}^{\Ga\Gb})^{T}+\Hat{X}_{\nu}^{\Ga\Gb}(\Hat{X}_{\mu}^{\Ga\Gb})^{T}=0,\quad\mu\not=\nu.
\]
It follows that $\{\Hat{X}_{1}^{\Ga\Gb},\ldots,\Hat{X}_{p}^{\Ga\Gb}\}\subset O(d)$.
Let $Y_{\mu}^{\Ga\Gb}=\Hat{X}_{\mu}^{\Ga\Gb}(\Hat{X}_{p}^{\Ga\Gb})^{T}$, $\mu=1,\ldots,p$. Then $Y_{p}^{\Ga\Gb}=I_{d}$ and
\begin{equation}
  \label{RH}
(Y_{\mu}^{\Ga\Gb})^{2}=-I_{d},\quad Y_{\mu}^{\Ga\Gb}+(Y_{\mu}^{\Ga\Gb})^{T}=0,\quad
Y_{\mu}^{\Ga\Gb}Y_{\nu}^{\Ga\Gb}+Y_{\nu}^{\Ga\Gb}Y_{\mu}^{\Ga\Gb}=0,\quad 1\le\mu<\nu\le p-1.  
\end{equation}
We consider two cases $r=2$ and $r>2$.

\subsubsection{Case $r>2$}
\label{sss:rg2}
When $r>2$ Remark~\ref{rem:jvnw} implies that $\mathfrak{J}$ can only be
$\Sym(\bb{R}^{r})$, for which $p=1$, $\mathfrak{H}(\bb{C}^{r})$, for which
$p=2$, or $\mathfrak{H}(\bb{H}^{r})$, for which $p=4$. The explicit
description of
the irreducible representations of these three algebras will be
written in terms of the canonical matrix representations of complex numbers
and quaternions. Specifically we define the maps
$\Gvf,\psi:\bb{C}\to\End_{\bb{R}}(\bb{R}^{2})$ and
$Q:\bb{H}\to\End_{\bb{R}}(\bb{R}^{4})$ as follows
\begin{equation}
  \label{phipsi}
\Gvf(x+iy)=\mat{x}{-y}{y}{x},\qquad\psi(x+iy)=\mat{x}{y}{y}{-x},  
\end{equation}
\begin{equation}
  \label{Qdef}
Q(q_{0}+iq_{1}+jq_{2}+kq_{3})=
\mat{\Gvf(q_{0}+iq_{1})}{-\psi(q_{2}+iq_{3})}{\psi(q_{2}+iq_{3})}{\Gvf(q_{0}+iq_{1})}.  
\end{equation}
These functions have the properties
\[
\Gvf(a)z=az,\quad\psi(a)z=a\bra{z},\quad Q(q)h=qh.
\]
Additionally
\[
\Gvf(a)^{T}=\Gvf(a),\quad\psi(a)^{T}=\psi(a),\quad Q(q)^{T}=Q(\bra{q}).
\]
For a collection of $d\times d$ matrices $M_{\Ga\Gb}$, $1\le\Ga\le r$,
$1\le\Gb\le r$ the notation $(M_{\Ga\Gb})$ stands for the $dr\times dr$ matrix
given in block-form by the $r\times r$ block-matrix, whose $\Ga\Gb$-blocks are
$d\times d$ matrices $M_{\Ga\Gb}$.
\begin{theorem}
  \label{th:irrep1}
  Up to an orthogonal conjugation a simple formally real special Jordan algebra
  $\mathfrak{J}$ with $r\ge 3$ has a unique irreducible representation by
  matrices in $\Sym(\bb{R}^{n})$, given explicitly by the following formulas
\begin{enumerate}
\item[(a)] $\mathfrak{J}=\Sym(\bb{R}^{r})$. Then $n=r$ and $\Phi(J)=J$, $J\in\Sym(\bb{R}^{r})$
\item[(b)] $\mathfrak{J}=\mathfrak{H}(\bb{C}^{r})$. Then $n=2r$ and
$
\mathfrak{H}(\bb{C}^{r})\ni J\mapsto\Phi(J)=(\varphi(J_{\Ga\Gb})).
$
\item[(c)] $\mathfrak{J}=\mathfrak{H}(\bb{H}^{r})$. Then $n=4r$ and
$
\mathfrak{H}(\bb{H}^{r})\ni J\mapsto\Phi(J)=(Q(J_{\Ga\Gb})).
$
\end{enumerate}
\end{theorem}
\begin{proof}
  Let $\bb{K}$ denote the division algebra $\bb{R}$, $\bb{C}$ or $\bb{H}$ in
  the definition of the algebra $\mathfrak{J}$. For an irreducible
  representation $\Phi:\mathfrak{J}\to\Sym(\bb{R}^{n})$ ($n=dr$) and $\Ga\not=\Gb$ let
$\Phi_{\Ga\Gb}:\bb{K}\to\CU^{\Ga\Gb}$, while $\Phi_{\Ga\Ga}(1)=I_{d}$, so that
$\Phi(J)=(\Phi_{\Ga\Gb}(J_{\Ga\Gb}))$. The maps $\Phi_{\Ga\Gb}$ satisfy
\[
\Phi_{\Ga\Gb}(q)^{T}=\Phi_{\Gb\Ga}(\bra{q}),\qquad
\Phi_{\Ga\Gb}(q)\Phi_{\Ga\Gb}(q)^{T}=|q|^{2}I_{d},\qquad
\Phi_{\Ga\Gb}(q)\Phi_{\Gb\Gg}(h)=\Phi_{\Ga\Gg}(qh),\quad\Ga\not=\Gg
\]
In particular, $\Phi_{\Ga\Gb}(1)\in O(d)$. Let
\[
R=\left[
  \begin{array}{cccc}
    I_{d}&&&0\\
    &\Phi_{12}(1)&&\\
    &&\ddots&\\
    0&&&\Phi_{1r}(1)
  \end{array}\right].
\]
Then $R\in O(n)$ and $\Psi(J)=R\Phi(J)R^{T}$ is an orthogonally equivalent
irreducible representation of $\mathfrak{J}$ in $\Sym(\bb{R}^{n})$, such that
$\Psi_{1\Ga}(1)=I_{d}$, $\Ga=1,\ldots,r$. Let $\Ga\not=\Gb$ and
$\Gb\not=1$. Then
\[
\Psi_{\Ga\Gb}(q)=\Psi_{1\Ga}(1)\Psi_{\Ga\Gb}(q)=\Psi_{1\Gb}(q).
\]
If in addition, $\Ga\not=1$ (which is possible \IFF $r\ge 3$), then
\[
\Psi_{\Ga\Gb}(q)=\Psi_{\Gb\Ga}(\bra{q})^{T}=\Psi_{1\Ga}(\bra{q})^{T}=\Psi_{\Ga1}(q)
\]
Hence, when $\Ga\not=\Gb$, $\Ga\not=1$, $\Gb\not=1$ we have
\[
\Psi_{\Ga\Gb}(q)\Psi_{\Ga\Gb}(h)=\Psi_{\Ga1}(q)\Psi_{1\Gb}(h)=\Psi_{\Ga\Gb}(qh).
\]
For any $\Gb\not=1$ there is $\Ga\not=1$ and $\Ga\not=\Gb$, since $r\ge 3$. Then
$\Psi_{1\Gb}(q)=\Psi_{\Ga\Gb}(q)$. Similarly, for any $\Ga\not=1$ there is
$\Gb\not=1$ and $\Ga\not=\Gb$, and it follows that
$\Psi_{\Ga1}(q)=\Psi_{\Ga\Gb}(q)$. Hence, for any $\Ga\not=\Gb$ the functions
$\Psi_{\Ga\Gb}$ are algebra homomorphisms. In particular, if $\Ga\not=\Gb$ we have
\[
\Psi_{\Gb\Ga}(q)=\Psi_{\Ga\Gb}(\bra{q})^{T}=|q|^{2}(\Psi_{\Ga\Gb}(\bra{q}))^{-1}=
\Psi_{\Ga\Gb}(|q|^{2}\bra{q}^{-1})=\Psi_{\Ga\Gb}(q).
\]
In particular, $\Psi_{21}(q)=\Psi_{12}(q)$. Also, for any $\Ga\not=1,2$
\[
\Psi_{1\Ga}(q)=\Psi_{\Ga1}(q)=\Psi_{\Ga2}(q)=\Psi_{12}(q).
\]
Finally, if $\Ga\not=\Gb$, $\Ga\not=1$, $\Gb\not=1$ then
\[
\Psi_{\Ga\Gb}(q)=\Psi_{1\Gb}(q)=\Psi_{12}(q).
\]
Thus, there is a division algebra homomorphism
$\Psi_{12}:\bb{K}\to\End_{\bb{R}}(\bb{R}^{d})$ such that
$\Psi_{12}(q)^{T}=\Psi_{12}(\bra{q})$ and $\Psi_{\Ga\Gb}(q)=\Psi_{12}(q)$ for
all $\Ga\not=\Gb$ and all $q\in\bb{K}$. The homomorphism $\Psi_{12}$ is
non-zero and hence injective (since $\bb{K}$ is a simple algebra over
$\bb{R}$). If the associative algebra
$\CU=\Psi_{12}(\bb{K})\subset\End_{\bb{R}}(\bb{R}^{d})$ has a proper invariant
subspace $V\subset\bb{R}^{d}$ then the subspace
\[
\underbrace{V\oplus\ldots\oplus V}_{r\text{ times}}\subset
\underbrace{\bb{R}^{d}\oplus\ldots\oplus\bb{R}^{d}}_{r\text{ times}}=\bb{R}^{n}
\]
is a proper invariant subspace of the representation $\Psi(\mathfrak{J})$. We
conclude that $\Psi_{12}$ must be an irreducible representation of $\bb{K}$
for an irreducible representation $\Psi$ of $\mathfrak{J}$. It remains to note
that there is a unique\footnote{For the sake of completeness this statement
  will be a consequence of our analysis of the case $r=2$. See
  Remarks~\ref{rem:c} and \ref{rem:q} below.} (up to an orthogonal conjugation) irreducible
representation $\Psi_{12}$ of $\bb{K}$ on $\bb{R}^{d}$ satisfying
$\Psi_{12}(q)^{T}=\Psi_{12}(\bra{q})$. It is given explicitly by the following
formulas
\begin{itemize}
\item If $\bb{K}=\bb{R}$ then $d=1$ and $\Psi_{12}(q)=q$.
\item If $\bb{K}=\bb{C}$ then $d=2$ and $\Psi_{12}(q)=\Gvf(q)$, given by (\ref{phipsi}).
\item If $\bb{K}=\bb{H}$ then $d=4$ and $\Psi_{12}(q)=Q(q)$, given by (\ref{Qdef}).
\end{itemize}
\end{proof}

\subsubsection{Case $r=2$}
\label{sss:re2}
When $r=2$ we have
\[
\Pi=\left\{\mat{\Gl I_{n/2}}{A}{A^{T}}{\mu I_{n/2}}:\{\Gl,\mu\}\subset\bb{R},\ A\in\CU\right\}.
\]
The conditions (\ref{RH}) on the basis $(I,Y_{1},\ldots,Y_{p-1})$ of
$\CU^{12}=\CU$ are both necessary and sufficient for $\Pi$ to be closed with
respect to the Jordan product (\ref{Jprod}). The irreducibility condition is
equivalent to the requirement that all matrices in $\CU$ have no common proper
invariant subspace. Indeed, if $\CU$ has a proper invariant subspace $V$ then
the space $\CV=\{(v_{1},v_{2}):\{v_{1},v_{2}\}\subset V\}$ is a proper
invariant subspace for $\Pi$. If $\CU$ has no proper invariant subspaces and
if $\CV$ is a proper invariant subspace of $\Pi$ then $(v,w)\in\CV$ implies
that $(v,0)=K_{1}(v,w)\in\CV$ and $(0,w)=K_{2}(v,w)\in\CV$, where
$K_{1}\in\Pi$ corresponds to $\Gl=1$, $\mu=0$, $A=0$ and $K_{2}\in\Pi$
corresponds to $\Gl=0$, $\mu=1$, $A=0$.  Then there are subspaces $V_{1}$ and
$V_{2}$ of $\bb{R}^{n/2}$ such that $\CV=\{(v_{1},v_{2}):v_{1}\in V_{1},\
v_{2}\in V_{2}\}$. The invariance of $\CV$ is then equivalent to the condition
that $\CU$ maps $V_{1}$ into $V_{2}$ and $V_{2}$ into $V_{1}$. However,
$I_{n/2}\in\CU$ and hence, if $v_{1}\in V_{1}$ then $v_{1}\in V_{2}$ and
conversely, if $v_{2}\in V_{2}$ then $v_{2}\in V_{1}$. It follows that
$V_{1}=V_{2}=V$. But then $V$ must be an invariant subspace for
$\CU$. Therefore, $V$ is either $\{0\}$ or $\bb{R}^{n/2}$, in which case $\CV$
is also either $\{0\}$ or $\bb{R}^{n}$.

The problem of identifying subspaces $\CU$ as described above is called a real
Radon-Hurwitz problem \cite{hurw22,radon22}, who studied it in connection with
the question of composition of quadratic forms. This problem has connections
to self-dual 2-forms \cite{deko03}, linearly independent vector fields on
spheres \cite{adams62}, system of hyperbolic conservation laws \cite{alp65},
Clifford algebras \cite{deoz07}, etc. In particular the matrices
$Y_{1},\ldots,Y_{p-1}$ in (\ref{RH}) are the generators of the
$2^{p-1}$-dimensional Clifford algebra $\CC\ell_{p-1,0}(\bb{R})$. We are
interested in an explicit and complete characterization of all possible
subspaces $\CU$ as above. While Hurwitz solved this problem in the complex
case, his solution is difficult to adapt to the real case. Here we present the
alternative solution based on the representation theory of finite groups
\cite{eckm42}.  Following \cite{eckm42} we let $G_{p}$ be a finite group with
generators $\Gve$, $a_{1},\ldots,a_{p-1}$, $p\ge 2$, satisfying the relations
\[
\Gve^{2}=1,\quad a_{k}^{2}=\Gve,\quad\Gve a_{k}=a_{k}\Gve,\quad
a_{k}a_{l}=\Gve a_{l}a_{k},\quad k,l=1,\ldots,p-1,\quad k\not=l.
\]
Then, the matrices $Y_{1},\ldots,Y_{p-1}$ satisfying (\ref{RH}) describe an
irreducible orthogonal representation of the group $G_{p}$ that sends $\Gve$
to $-I_{n/2}$. Eckmann \cite{eckm42} has computed the number, type, and
dimensions of such irreps of $G_{p}$ (see Appendix~\ref{app:eck}). His results
are summarized in the following theorem.
\begin{theorem}[Eckmann]
  \label{th:Eckmann}
All non-isomorphic irreducible real representations of $G_{p}$ that map $\Gve$
to $-I$ are characterized as follows.
\begin{itemize}
\item $p=1,7\mod 8$. Then the unique real irrep $V_{p}$ is a representation of real type
and  $d(p)=\dim V_{p}=2^{\frac{p-1}{2}}$. 
\item $p=3,5\mod 8$. Then the unique real irrep $V_{p}$ is a representation of
  quaternionic type 
and  $d(p)=\dim V_{p}=2^{\frac{p+1}{2}}$. 
\item $p=2,6\mod 8$. Then the unique real irrep $V_{p}$ is a representation of complex
  type 
and $d(p)=\dim V_{p}=2^{\frac{p}{2}}$. 
\item $p=0\mod 8$. Then there are two distinct irreps $V_{p}^{+}$ and $V_{p}^{-}$
  both of real type and  $d(p)=\dim V_{p}^{\pm}=2^{\frac{p-2}{2}}$. 
\item $p=4\mod 8$. Then there are two distinct irreps $V_{p}^{+}$ and $V_{p}^{-}$
  both of quaternionic type and  $d(p)=\dim V_{p}^{\pm}=2^{\frac{p}{2}}$. 
\end{itemize}
\end{theorem}
We note that the function $d(p)$ given explicitly in Theorem~\ref{th:Eckmann}
can also be defined recursively by
\begin{equation}
  \label{dp}
  d(p+8)=16d(p),\qquad d(1)=1,\ d(2)=2,\ d(3)=d(4)=4,\ d(5)=d(6)=d(7)=d(8)=8.
\end{equation}
It remains to construct the representations $V_{p}$ explicitly.  It is sufficient
to indicate the images of $a_{1},\ldots,a_{p-1}$ (since it is required that
$\Gve\mapsto-I_{d(p)}$). To describe the answer we need to introduce the
following notation.
\begin{equation}
  \label{HatQdef}
\Hat{Q}(q_{0}+iq_{1}+jq_{2}+kq_{3})=
\mat{\Gvf(q_{0}+iq_{1})}{\Gvf(q_{2}+iq_{3})}{-\Gvf(q_{2}-iq_{3})}{\Gvf(q_{0}-iq_{1})},  
\end{equation}
where $\Gvf$ was given in (\ref{phipsi}). The function $\Hat{Q}$ satisfies
\[
\Hat{Q}(q_{1})\Hat{Q}(q_{2})=\Hat{Q}(q_{1}q_{2}),\quad\Hat{Q}(q)^{T}=\Hat{Q}(\bra{q}),\quad
\Hat{Q}(q_{1})Q(q_{2})=Q(q_{2})\Hat{Q}(q_{1}).
\]
For $\{q,h\}\subset\bb{H}$ we also define
\begin{equation}
  \label{Odef}
 \bb{O}(q,h)=\mat{Q(q)}{\Hat{Q}(h)}{-\Hat{Q}(\bra{h})}{Q(\bra{q})},
\end{equation}
where the maps $Q$ and $\Hat{Q}$ are defined in (\ref{Qdef}) and
(\ref{HatQdef}), respectively.
The map $\bb{O}$ has the following properties
\[
\bb{O}(q,h)^{T}=\bb{O}(\bra{q},-h),\quad\bb{O}(q,h)\bb{O}(q,h)^{T}=(|q|^{2}+|h|^{2})I_{8}.
\]

For $2\le p\le 9$ we have the following explicit representations, which are
slightly modified versions of the ones in \cite{deko03}.
\begin{reflist}
  \label{lst:irrep}
\end{reflist}
\begin{itemize}
\item $p=2$, $d(p)=2$, $\rho_{2}(a_{1})=\Gvf(i)$.
  \begin{remark}
    \label{rem:c}
The uniqueness of the irreducible representation $V_{2}$ implies that up to an
orthogonal conjugation the irreducible representation of $\bb{C}$ on
$\End_{\bb{R}}(\bb{R}^{d})$ satisfies $d=2$ and is given by $\bb{C}\ni c\mapsto\Gvf(c)$.
  \end{remark}
\item $p=3$, $d(p)=4$, $\rho_{3}(a_{1})=Q(i)$, $\rho_{3}(a_{2})=Q(j)$
\item $p=4$, $d(p)=4$. There are two non-isomorphic representations:
\[
\rho_{4}^{\pm}(a_{1})=\pm Q(i),\qquad\rho_{4}^{\pm}(a_{2})=\pm Q(j),\qquad
\rho_{4}^{\pm}(a_{3})=\pm Q(k).
\]
Indeed,
\[
\rho_{4}^{+}(a_{1}a_{2})=\rho_{4}^{+}(a_{3}),\qquad
\rho_{4}^{-}(a_{1}a_{2})=-\rho_{4}^{-}(a_{3}).
\]
\begin{remark}
  \label{rem:q}
Let $\Psi_{0}:\bb{H}\to\End_{\bb{R}}(\bb{R}^{d})$ be a representation of
$\bb{H}$. If we define $\rho(\Gve)=-I_{d}$, $\rho(a_{1})=\Psi_{0}(i)$, $\rho(a_{2})=\Psi_{0}(j)$,
$\rho(a_{3})=\Psi_{0}(k)$ then $\rho$ will be a representation of $G_{4}$ on
$\bb{R}^{d}$. Clearly, $\rho$ is irreducible \IFF $\Psi_{0}$ is
irreducible. Thus, $d=4$ and, up to the orthogonal conjugation, either
\[
\Psi_{0}(i)=Q(i),\qquad\Psi_{0}(j)=Q(j),\qquad\Psi_{0}(k)=Q(k),
\]
or
\[
\Psi_{0}(i)=-Q(i),\qquad\Psi_{0}(j)=-Q(j),\qquad\Psi_{0}(k)=-Q(k).
\]
However, the latter choice results in
$\Psi_{0}(ij)=-\Psi_{0}(i)\Psi_{0}(j)$. Hence, up to an orthogonal
conjugation $\Psi_{0}(q)=Q(q)$.
\end{remark}
\item $p=5$, $d(p)=8$,
\[
\rho_{5}(a_{1})=\bb{O}(0,1),\quad\rho_{5}(a_{2})=\bb{O}(0,i),\quad\rho_{5}(a_{3})=\bb{O}(0,j),\quad\rho_{5}(a_{4})=\bb{O}(0,k).
\]
\item $p=6$, $d(p)=8$,
\[
\rho_{6}(a_{1})=\bb{O}(0,i),\quad\rho_{6}(a_{2})=\bb{O}(0,j),\quad\rho_{6}(a_{3})=\bb{O}(0,k),
\]
\[
\rho_{6}(a_{4})=\bb{O}(i,0),\quad\rho_{6}(a_{5})=\bb{O}(j,0).
\]
\item $p=7$, $d(p)=8$,
\[
\rho_{7}(a_{1})=\bb{O}(0,i),\quad\rho_{7}(a_{2})=\bb{O}(0,j),\quad\rho_{7}(a_{3})=\bb{O}(0,k),
\]
\[
\rho_{7}(a_{4})=\bb{O}(i,0),\quad\rho_{7}(a_{5})=\bb{O}(j,0),\quad\rho_{7}(a_{6})=\bb{O}(k,0).
\]
\item $p=8$, $d(p)=8$: There are two non-isomorphic representations:
  $\rho_{8}^{-}(a_{i})=-\rho_{8}^{+}(a_{i})$, $i=1,\ldots,7$
\[
\rho_{8}^{-}(a_{1})=\bb{O}(0,1),\quad\rho_{8}^{-}(a_{2})=\bb{O}(0,i),\quad\rho_{8}^{-}(a_{3})=\bb{O}(0,j),\quad\rho_{8}^{-}(a_{4})=\bb{O}(0,k),
\]
\[
\rho_{8}^{-}(a_{5})=\bb{O}(i,0),\quad\rho_{8}^{-}(a_{6})=\bb{O}(j,0),\quad\rho_{8}^{-}(a_{7})=\bb{O}(k,0).
\]
The irrep $\rho_{8}^{-}$ is not isomorphic to $\rho_{8}^{+}$ because
\[
\rho_{8}^{+}(a_{2}a_{3}a_{4}a_{5}a_{6}a_{7})=\rho_{8}^{+}(a_{1}),
\]
whereas
\[
\rho_{8}^{-}(a_{2}a_{3}a_{4}a_{5}a_{6}a_{7})=-\rho_{8}^{-}(a_{1}),
\]
\item $p=9$, $d(p)=16$.
\[
\rho_{9}(a_{1})=\psi(i)\otimes\bb{O}(i,0),\quad
\rho_{9}(a_{2})=\psi(i)\otimes\bb{O}(j,0),\quad
\rho_{9}(a_{3})=\psi(i)\otimes\bb{O}(k,0),
\]
\[
\rho_{9}(a_{4})=\psi(i)\otimes\bb{O}(0,i),\quad
\rho_{9}(a_{5})=\psi(i)\otimes\bb{O}(0,j),\quad
\rho_{9}(a_{6})=\psi(i)\otimes\bb{O}(0,k),
\]
\[
\rho_{9}(a_{7})=\psi(i)\otimes\bb{O}(0,1),\quad
\rho_{9}(a_{8})=\Gvf(i)\otimes\bb{O}(1,0).
\]
\end{itemize}
For $m_{1}\times n_{1}$ matrix $A$ and $m_{2}\times n_{2}$ matrix $B$ the
tensor product notation $A\otimes B$ denotes $m_{1}m_{2}\times n_{1}n_{2}$
matrix written in block-form as
\[
A\otimes B=\left[
  \begin{array}{ccc}
    a_{11}B &\ldots&a_{1n_{1}}B\\
    \ldots &\ldots&\ldots\\
    a_{m_{1}1}B&\ldots&a_{m_{1}n_{1}}B
  \end{array}
\right].
\]

Suppose that representations $\rho_{p}$ of $G_{p}$ have
been constructed for $2\le p\le 9$. For $p\ge 10$ we define
\begin{equation}
  \label{recursion}
  \rho_{p}(a_{i})=\rho_{9}(a_{i})\otimes I_{d(p-8)},\ i=1,\ldots,8,\quad
\rho_{p}(a_{i})=\BB\otimes\rho_{p-8}(a_{i}),\ i=9,\ldots,p-1,
\end{equation}
where $\BB=\psi(1)\otimes I_{8}\in\Sym(\bb{R}^{16})$. 
We have $\BB^{2}=I_{16}$.
Therefore, $\rho_{p}(a_{i})^{2}=-I_{p}$ and
$\rho_{p}(a_{i})\rho_{p}(a_{j})=-\rho_{p}(a_{j})\rho_{p}(a_{i})$ if
$i\not=j$ and both $i$ and $j$ are either below 9 or above 8. We need to check
the anticommutativity property for $i\le 8$ and $j\ge 9$:
\[
\rho_{9}(a_{i})\BB\otimes\rho_{p-8}(a_{j})=-\BB\rho_{9}(a_{i})\otimes\rho_{p-8}(a_{j}).
\]
This will be satisfied if
\begin{equation}
  \label{anticomm}
\rho_{9}(a_{i})\BB=-\BB\rho_{9}(a_{i}).  
\end{equation}
We can verify (\ref{anticomm}) explicitly via the formulas for $\rho_{9}$ below.

In addition to the explicit representations $V_{p}$ it is also convenient to
have an explicit form of the subspaces 
\[
\CW_{p}=\Span\{\rho_{p}(a_{1}),\ldots,\rho_{p}(a_{p-1})\}\subset\Skew(\bb{R}^{d(p)}),\qquad
\CU=\CU_{p}=\bb{R}I_{d(p)}\oplus\CW_{p}.
\]
\begin{reflist}
  \label{lst:UW}
\end{reflist}
\begin{itemize}
\item $p=1$, $\CW_{1}=\{0\}\subset\bb{R}$, $\CU_{1}=\bb{R}$
\item $p=2$,
\[
\CW_{2}=\bb{R}\Gvf(i)=\Skew(\bb{R}^{2}),
\]
\[
\CU_{2}=\{\Gvf(z):z\in\bb{C}\}.
\]
\item $p=3$,
\[
\CW_{3}=\{Q(q):q\in\bb{H},\ q=q_{1}i+q_{2}j\}.
\]
\[
\CU_{3}=\{Q(q):q\in\bb{H},\ q=q_{0}+q_{1}i+q_{2}j\}.
\]
\item $p=4$,
\[
\CW_{4}=\{Q(q):q\in\bb{H},\ \re(q)=0\},
\]
\[
\CU_{4}=\{Q(q):q\in\bb{H}\}.
\]
\item $p=5$,
\[
\CW_{5}=\{\bb{O}(0,h):h\in\bb{H}\},
\]
\[
\CU_{5}=\{\bb{O}(\Gl,h):h\in\bb{H},\ \Gl\in\bb{R}\}.
\]
\item $p=6$,
\[
\CW_{6}=\left\{\bb{O}(q,h):\{q,h\}\subset\bb{H},\ \re(h)=0,\ q=q_{1}i+q_{2}j\right\}.
\]
\[
\CU_{6}=\left\{\bb{O}(q,h):\{q,h\}\subset\bb{H},\ \re(h)=0,\ q=q_{0}+q_{1}i+q_{2}j\right\}.
\]
\item $p=7$,
\[
\CW_{7}=\left\{\bb{O}(q,h):\{q,h\}\subset\bb{H},\ \re(q)=0,\ \re(h)=0\right\},
\]
\[
\CU_{7}=\left\{\bb{O}(q,h):\{q,h\}\subset\bb{H},\ \re(h)=0\right\}.
\]
\item $p=8$,
\[
\CW_{8}=\left\{\bb{O}(q,h):\{q,h\}\subset\bb{H},\ \re(q)=0\right\},
\]
\[
\CU_{8}=\left\{\bb{O}(q,h):\{q,h\}\subset\bb{H}\right\}.
\]
\end{itemize}
The formula (\ref{recursion}) results in
\begin{equation}
  \label{Wp}
\CW_{p}=\left\{\mat{I_{8}\otimes A}{\bb{O}(q,h)\otimes I_{d(p-8)}}
{-\bb{O}(q,h)^{T}\otimes I_{d(p-8)}}{-I_{8}\otimes A}:
A\in\CW_{p-8},\ \{q,h\}\subset\bb{H}\right\},\quad p\ge 9,  
\end{equation}
where the spaces $\CW_{1},\ldots\CW_{8}$ are given in the List~\ref{lst:UW}
and $\bb{O}(q,h)$ is defined in (\ref{Odef}).

We remark that the formulas for $\CW_{p}$ and $\CU_{p}$, $p=1,\ldots,8$ are
somewhat arbitrary, since conjugation by any orthogonal transformation would produce
equally valid formulas. However, when $p=1,2,4$ and 8 the spaces $\CW_{p}$ and
$\CU_{p}$ are $O(d(p))$-invariant and therefore represent the canonical
forms. The formula (\ref{Wp}) implies that all spaces $\CW_{2^{s}}$ and
$\CU_{2^{s}}$ are canonical.

It is important to note that the two different representations of $G_{p}$ when $p=0\mod 4$
result in different representations of the spin factors $\mathfrak{S}_{N}$,
when $N=2\mod 4$ ($N\ge 6$), even though the images of $\Pi_{\mathfrak{S}_{N}}$
of $\mathfrak{S}_{N}$ under both representations are the same:
\begin{equation}
  \label{spinrep}
  \Pi_{\mathfrak{S}_{N}}=\left\{\mat{\Gl I_{d(N-2)}}{A}{A^{T}}{\mu I_{d(N-2)}}:
\{\Gl,\mu\}\subset\bb{R},\ A\in\CU_{N-2}=\bb{R}I_{d(N-2)}\oplus\CW_{N-2}\right\},
\end{equation}
where $\CW_{p}$ is given by (\ref{Wp}).
The existence
of the two non-isomorphic representations of $G_{p}$ is reflected in the
existence of the map $T:\Pi_{\mathfrak{S}_{N}}\to\Pi_{\mathfrak{S}_{N}}$,
which maps $\rho^{+}$ to $\rho^{-}$. As such it is a Jordan algebra
automorphism that cannot be written as an orthogonal conjugation. Conversely,
every Jordan algebra automorphism maps one representation of $G_{p}$ into
another. For those $p$ for which such a representation is unique the
automorphism must be an orthogonal conjugation. The Jordan algebra automorphism
$T:\Pi_{\mathfrak{S}_{N}}\to\Pi_{\mathfrak{S}_{N}}$ can be written explicitly as
\begin{equation}
  \label{Jauto}
T\mat{\Gl I_{d(N-2)}}{A}{A^{T}}{\mu I_{d(N-2)}}=
\mat{\Gl I_{d(N-2)}}{A^{T}}{A}{\mu I_{d(N-2)}},\qquad A\in\CU_{N-2},\ \{\Gl,\mu\}\subset\bb{R}.
\end{equation}

Thus, we have proved the following characterization of all irreducible
representations of spin factors $\mathfrak{S}_{N}$, $N\ge 3$
($\mathfrak{S}_{1}=\bb{R}$ and $\mathfrak{S}_{2}$ is not simple) in $\Sym(\bb{R}^{n})$.
\begin{theorem}
\label{th:irred}
Each spin factor $\mathfrak{S}_{N}$, $N\ge 3$ is represented by a
unique (up to the orthogonal conjugation) irreducible subalgebra
$\Pi_{\mathfrak{S}_{N}}\subset\Sym(\bb{R}^{2d(N-2)})$ given by (\ref{spinrep}). 
Moreover, each Jordan algebra automorphisms of
$\Pi_{\mathfrak{S}_{N}}$ can be represented by an orthogonal conjugation, unless
$N=2\mod 4$ ($N\ge 3$), in which case  each Jordan algebra automorphisms of
$\Pi_{\mathfrak{S}_{N}}$ can be represented by a composition of the map $T$
given by (\ref{Jauto}) and an orthogonal conjugation.
\end{theorem}

\section{Structure theorem and \completeness}
\setcounter{equation}{0}
\label{sec:struct}
We can now summarize all our results and describe explicitly, up to an
orthogonal conjugation all Jordan subalgebras of $\Sym(\bb{R}^{n})$.
\begin{theorem}[Jordan subalgebras of $\Sym(\bb{R}^{n})$]~
\label{th:jast}
  \begin{enumerate}
  \item[(i)] Let $\Pi$ be a Jordan subalgebra of $\Sym(\bb{R}^{n})$. Then 
there exists an o.n.b. of $\bb{R}^{n}$ in which
\[
\Pi=\left\{\left[
  \begin{array}{cccc}
    A_{1}&&&0\\
    &\ddots&&\\
    &&A_{m}&\\
    0&&&0
  \end{array}\right]:A_{1}\in\Pi_{1},\ldots,A_{m}\in\Pi_{m}\right\},
\]
where each $\Pi_{\Ga}$ is a simple non-singular Jordan subalgebra of $\Sym(\bb{R}^{d_{\Ga}})$, for
some $d_{\Ga}\ge 1$ for which $d_{1}+\ldots+d_{m}\le n$. 
\item[(ii)] Let $\Pi$ be a simple non-singular Jordan subalgebra of
  $\Sym(\bb{R}^{n})$. Then there exits an o.n.b. in $\bb{R}^{n}$ in which $\Pi$ has one of the
  following forms
  \begin{itemize}
  \item[(a)] $\Pi=\{I_{n/r}\otimes A:A\in\Sym(\bb{R}^{r})\}$, $r\ge 1$;
  \item[(b)] $\Pi=\{I_{n/2r}\otimes A:A\in\mathfrak{H}(\bb{C}^{r})\}$, $r\ge 2$;
  \item[(c)] $\Pi=\{I_{n/4r}\otimes A:A\in\mathfrak{H}(\bb{H}^{r})\}$, $r\ge 2$;
  \item[(d)] $\Pi=\{I_{n/2d(N-2)}\otimes A:A\in\Pi_{\mathfrak{S}_{N}}\}$, $N=5,7,8,\ldots$, 
    where $\Pi_{\mathfrak{S}_{N}}$ is given by (\ref{spinrep}) and function
    $d(\cdot)$ is defined in (\ref{dp});
  \item[(e)] $\Pi=\left\{\mat{I_{s_{1}}\otimes A}{0}{0}{I_{s_{2}}\otimes TA}:
A\in\Pi_{\mathfrak{S}_{N}}\right\}$, $s_{1}+s_{2}=n/2d(N-2)$, $s_{1}>0$,
$s_{2}>0$, $N=2\mod 4$, $N\ge 6$, where the Jordan automorphism
$T:\Pi_{\mathfrak{S}_{N}}\to\Pi_{\mathfrak{S}_{N}}$ is given by (\ref{Jauto}).
  \end{itemize}
 \end{enumerate}
\end{theorem}
The explicit characterization of all Jordan subalgebras of $\Sym(\bb{R}^{n})$
in Theorem~\ref{th:jast} allows us to answer a question about \completeness\ of
a subalgebra $\Pi\subset\Sym(\bb{R}^{n})$. By part~(i) of
Theorem~\ref{th:jast} we can write $\Pi=\Pi_{1}\oplus\ldots\oplus\Pi_{m}$,
where $\Pi_{\Ga}$ is a simple non-singular subalgebra of
$\Sym(\bb{R}^{d_{\Ga}})$. The algebra $\Pi$ will be \complete\  \IFF each of the
algebras $\Pi_{\Ga}$ is orthogonally equivalent to one of the algebras in
cases (a), (b) or (c) in part~(ii) of Theorem~\ref{th:jast}.

We remark on a single exception to the statement that \completeness\ is
determined by the isomorphism class of $\Pi$. The algebras $\Pi$ in
part~(ii)(e) $N=6$ and part~(ii)(c) $r=2$ are 
isomorphic to $\mathfrak{H}(\bb{H}^{2})$. However, the algebras in
part~(ii)(c) $r=2$ are \complete, while the ones in part~(ii)(e) $N=6$ are not.

Suppose now that we are looking for all subalgebras of $\Sym(\bb{R}^{n})$ that
are closed not only with respect to the product (\ref{Jprod}) but also with
respect to the products (\ref{GJprod}),
where the subspace $\CA$ of
multiplications is spanned by finitely many matrices
$A_{1}=I_{n},A_{2},\ldots,A_{s}$. Let $\Pi$ be a non-singular
$\CA$-multialgebra. Then, by Lemma~\ref{lem:nonsing},
$I_{n}\in\Pi$. Therefore, $A_{j}=I_{n}A_{j}I_{n}\in\Pi$, $j=2,\ldots,s$, and
hence, each multiplication $X*_{A_{j}}Y$ is a mutation of the standard
multiplication (\ref{Jprod}) \cite{koech99}. We
first choose a basis in which $\Pi$ has the form described in the Structure
Theorem~\ref{th:jast}. In that basis the submatrices of $A_{2},\ldots,A_{s}$
corresponding to $\CN(\Pi)^{\perp}$ must be block-diagonal, with each diagonal
block being a member of a simple component $\Pi_{\Ga}$ of $\Pi$. That is also
sufficient, since the triple product
$K_{\Ga}A_{\Ga}L_{\Ga}+L_{\Ga}A_{\Ga}K_{\Ga}$ always belongs to $\Pi_{\Ga}$. If we
exclude the special case when the $\Ga\Ga$-block $\CA_{\Ga}$ of $\CA$ has the form
$\CA_{\Ga}=I_{k}\otimes\CA_{0}$ we will exclude all cases in which the
simple algebra $\Pi_{\Ga}$ is reducible. In particular, this would exclude all
part~(ii)(e) cases. Hence, the question of \completeness\ of
$\CA$-multialgebras reduces to the question of \completeness\ of
$\CA$-multialgebras on the invariant subspaces of $\CA$. If $\CA$ has no
invariant subspaces then $\Pi$ must be irreducible. An irreducible algebra is
always \complete\  if $n$ is not a power of 2 or is less than 8.

Even though possibilities like $\CA_{\Ga}=I_{k}\otimes\CA_{0}$ cannot be ruled
out in general, we may try to understand when in\complete\ algebras can arise
generically. The subspace of $\Sym(\bb{R}^{n})$ of all matrices whose upper
left $8\times 8$ submatrix is in $\Pi_{\mathfrak{S}_{5}}$ has co-dimension
31. Therefore, the space $\CV_{0}$ of $s$-tuples of symmetric matrices (one of
which is $I_{n}$), all of whose upper left $8\times 8$ submatrices is in
$\Pi_{\mathfrak{S}_{5}}$ has co-dimension $31(s-1)$. The dimension of the
$O(n)$ orbit of a generic $s$-tuple of matrices is $n(n-1)/2$. The subspace
$\CV_{0}$ will intersect this $O(n)$ orbit generically only if $n(n-1)\ge
62(s-1)$. For example, for a generic 5-dimensional space $\CA$ (containing
$I_{n}$), all Jordan multialgebras will be \complete, if $n<17$. These
dimensional considerations explain the reason for \completeness\ in all of the
examples in
\cite{gms00,grab03,grab09,hegg-thesis12}. Conjecture~\ref{conj:compl} suggests
that another way to eliminate failure of \completeness\ may be to restrict the
Jordan multialgebras to $SO(3)$-invariant ones.

\section{Lamination exact relation that is not closed under homogenization}
\setcounter{equation}{0}
\label{sec:cex}
Returning to the physical example of multifield response composite materials
in Section~\ref{sec:ex} we see that the smallest $n$ for which the in\complete\  
Jordan subalgebra of $\Sym(\bb{R}^{n})$ provides an example on in\complete\  
Jordan $\CA$-multialgebra is $n=8$. However, the quaternionic formalism of
Section~\ref{sec:iralg} suggests an example of in\complete\  
Jordan $\CA$-multialgebra when $n=4$\footnote{We have identified all solutions of
(\ref{2dJMA}) for $n=2$ by hand and verified that all the Jordan
$\CA$-multialgebras are complete in this case. The case $n=3$ remains unexplored.}. 
Let
\begin{equation}
  \label{incompex}
  \CL=\{iQ(q):q\in\bb{H},\ \re(q)=0\}\subset\mathfrak{H}(\bb{C}^{4}),\qquad
\CM=\{aI_{4}:a\in\bb{C}\}\subset\Sym(\bb{C}^{4}),
\end{equation}
where $Q(q)$ is defined in (\ref{Qdef}). Let us verify that the 3-chain
condition (\ref{3chain}) fails.  Indeed, let $c=i$, $X_{1}=iQ(q_{1})$,
$X_{2}=iQ(q_{2})$, $X_{2}=iQ(q_{3})$, where it will be convenient to identify
the purely imaginary quaternions $q_{1}$, $q_{2}$ and $q_{3}$ with vectors in
$\bb{R}^{3}$. Let us assume that $\{q_{1},q_{2},q_{3}\}$ form a basis of
$\bb{R}^{3}$, so that the mixed product
$(q_{1},q_{2},q_{3})=q_{1}\cdot(q_{2}\times q_{3})$ is non-zero.  We compute
\[
X_{1}X_{2}^{T}X_{3}=iQ(q_{1}q_{2}q_{3}).
\]
Then 
\[
iX_{1}X_{2}^{T}X_{3}-iX_{3}X_{2}^{T}X_{1}=Q(2(q_{1},q_{2},q_{3}))=2(q_{1},q_{2},q_{3})I_{4}\not\in\CL.
\]
If we add real multiples of $I_{4}$ to $\CL$, we will obtain a completion of
our incomplete Jordan $\CA$-multialgebra:
\begin{equation}
  \label{compex}
\bra{\CL}=\{\Ga I_{4}+iQ(q):\Ga\in\bb{R},\ q\in\bb{H},\ \re(q)=0\},\qquad
\bra{\CM}=\CM=\{aI_{4}:a\in\bb{C}\},
\end{equation}
Indeed, the Jordan $\CA$-multialgebra (\ref{compex}) consists of all
elements in $\Sym(\CT)$ that belong to the associative $\CA$-multialgebra
\begin{equation}
  \label{assocex}
  \CL'=\CM'=\{aQ(q):a\in\bb{C},\ q\in\bb{H}\}.
\end{equation}
It is easy to check that the  associative $\CA$-multialgebra (\ref{assocex})
is symmetric in the sense of Definition~\ref{def:sja}.

In \cite{grab03} it was shown that the 3-chain relation property is necessary
for the $SO(2)$-invariant Jordan multialgebra to correspond to an exact
relation. Hence, the incomplete Jordan multialgebra (\ref{incompex}) provides
the first example of a rotationally invariant lamination exact relation in 2D
that is not stable under homogenization. For the sake of reference we compute
the image of both (\ref{incompex}) and its completion (\ref{compex}) in
physical variables. In order to formulate the results it will be convenient to
identify $\CT=\bb{R}^{2}\otimes\bb{R}^{4}$ with $\bb{H}^{2}$, via the isomorphisms
\[
\CT=\bb{R}^{2}\otimes\bb{R}^{4}\cong\bb{R}^{4}\oplus\bb{R}^{4}\cong
\bb{H}\oplus\bb{H}\cong\bb{H}^{2},
\]
where we have identified $\bb{R}^{4}$ with $\bb{H}$. The quaternionic
materials are identified with $\CH^{+}(\bb{H}^{2})$|the set of positive
definite quaternionic-Hermitean $2\times 2$ matrices
\[
\BL=\mat{\Gl}{h}{\bra{h}}{\mu},\qquad\Gl>0,\ \mu>0,\ h\in\bb{H},\ \det\BL=\Gl\mu-|h|^{2}>0.
\]
If
\[
\BE=\vect{q_{1}}{q_{2}}\in\bb{H}^{2}=\CT,\qquad
\BJ=\vect{p_{1}}{p_{2}}\in\bb{H}^{2}=\CT,\qquad\BJ=\BL\BE,
\]
then
\[
\vect{p_{1}}{p_{2}}=\vect{\Gl q_{1}+hq_{2}}{\bra{h}q_{1}+\mu q_{2}}.
\]
We may also represent the multifield response of quaternionic materials by the
conventional $8\times 8$ real symmetric matrix
\[
\BL=\mat{\Gl I_{4}}{Q(h)}{Q(\bra{h})}{\mu I_{4}},
\]
where $Q(q)$ is defined in (\ref{Qdef}).

Finally, the lamination exact relation that is not stable under homogenization
corresponding to (\ref{incompex}) is given by
\begin{equation}
  \label{myex}
\bb{M}=\{\BL\in\CH^{+}(\bb{H}^{2}):\det\BL=1\}.  
\end{equation}
Obviously, $\det\BL$ can be any positive constant in the definition of
$\bb{M}$. The constant is set to 1 for simplicity.

\section{Acknowledgement}
This material is based upon work supported by the National Science Foundation
under Grant No. 1008092.

\appendix

\section{Summary of \cite{eckm42}}
\setcounter{equation}{0} 
\label{app:eck}
Let $G_{p}$ be a finite group with generators
$a_{1},\ldots,a_{p-1}$, $p\ge 2$ and $\Gve$ satisfying the relations
\[
\Gve^{2}=1_{G_{p}},\quad a_{k}^{2}=\Gve,\quad\Gve a_{k}=a_{k}\Gve,\quad
a_{k}a_{l}=\Gve a_{l}a_{k},\quad k,l=1,\ldots,p-1,\quad k\not=l.
\]
Let $S={i_{1},\ldots,i_{k}}$ be a subset of $\{1,\dots,p-1\}$, where $1\le
i_{1}<i_{2}<\ldots<i_{k}\le p-1$. Let $a_{S}=a_{i_{1}}a_{i_{2}}\ldots
a_{i_{k}}$, where $a_{\emptyset}=1_{G_{p}}$. Then $G_{p}=\{a_{S},\Gve
a_{S}:S\subset\{1,\ldots,p-1\}\}$. Hence, $|G_{p}|=2^{p}$.
The first observation is that the commutator subgroup of
$G_{p}$ is $K=\{1_{G_{p}},\Gve\}$. It follows that $G_{p}$ has exactly $|G_{p}/K|=2^{p-1}$
non-isomorphic complex 1D representations. In each of them $\Gve$ gets sent to
1, since $\Gve$ belongs to a commutator subgroup. 

The second observation is that it is
easy to list all conjugacy classes of $G_{p}$ explicitly. They are
\[
\begin{cases}
  \{1_{G_{p}}\},\ \{\Gve\},\ \{a_{S},\Gve a_{S}\},\ 
  S\subsetneqq\{1,\ldots,p-1\},\ \{a_{\{1,\ldots,p-1\}}\},\
  \{\Gve a_{\{1,\ldots,p-1\}}\}, & p\text{ is even},\\
  \{1_{G_{p}}\},\ \{\Gve\},\ \{a_{S},\Gve a_{S}\},\ 
  S\subset\{1,\ldots,p-1\},\ S\not=\emptyset,\ 
  \{a_{\{1,\ldots,p-1\}},\Gve a_{\{1,\ldots,p-1\}}\}, & p\text{ is odd}.
\end{cases}
\]
The total number of non-isomorphic complex irreps of $G_{p}$ equals to the
number of conjugacy classes of $G_{p}$. We see that when $p$ is odd there is
exactly one non-1D irrep, while when $p$ is even there are exactly two.  The
sum of squares of the dimensions of all the irreps of $G_{p}$ equals to the
order of $G_{p}$. Hence, when $p$ is odd the dimension $d$ of the non-1D irrep
is $d=2^{\frac{p-1}{2}}$. When $p$ is even we denote the dimensions of the two
non-1D irreps by $d_{1}$ and $d_{2}$. Recall that $d_{1}$ and $d_{2}$ must
divide the order of the group. Hence, $d_{1}=2^{\Ga}$, $d_{2}=2^{\Gb}$ and
$2^{2\Ga}+2^{2\Gb}+2^{p-1}\cdot 1^{2}=2^{p}$. Thus,
$d_{1}=d_{2}=2^{\frac{p-2}{2}}$. We remark that $\Gve$ cannot get sent to the
identity matrix $I$ in a non-1D irrep, since in that case the commutator of
$G_{p}$ gets mapped to $I$ and all images of elements of $G_{p}$ would
commute with one another. Since the image of $\Gve$ must commute with the
images of all elements in the group, it must be mapped (by Schur's lemma) into
a multiple of the identity $\Gl I$ for some $\Gl\in\bb{C}$. However,
$\Gl^{2}=1$, since $\Gve^{2}=1_{G_{p}}$. Therefore, $\Gve$ must be sent to
$-I$.

The third observation is that $g^{2}=1_{G_{p}}$ or $\Gve$ for any $g\in
G_{p}$. In fact, $(\Gve a_{S})^{2}=a_{S}^{2}=\Gve^{r(r+1)/2}$, where
$r=|S|$. Therefore, $a_{S}^{2}$ gets mapped to $(-1)^{r(r+1)/2}I$.
This allows an explicit computation of the Frobenius-Schur indicator
\[
S=\nth{|G_{p}|}\sum_{g\in G_{p}}\chi(g^{2}),
\]
where $\chi$ is the character of the representation. The result is
\[
S={\rm sign}\left(\cos\left(\frac{\pi p}{4}\right)\right).
\]
This shows that the non-1D irreps of $G_{p}$ are of real type, when
$p=0,1,7\mod 8$, complex type, when $p=2,6\mod 8$ and quaternionic type when
$p=3,4,5\mod 8$. In particular, when $p=2,6\mod 8$ the representations $U$ and
$\bra{U}$ are not isomorphic and hence exhaust the list of two non-isomorphic
irreps for $G_{p}$. 

Thus, in order to obtain real irreps of $G_{p}$ in which $\Gve$ gets sent to
$-I$ we take real parts of the representations $U$ for $p=0,1,7\mod 8$, real
parts of the representations $U\oplus U$ for $p=3,4,5\mod 8$ and real
parts of the representations $U\oplus\bra{U}$ for $p=2,6\mod
8$. Theorem~\ref{th:Eckmann} summarizes the results.


\end{document}